\documentclass[11pt,a4paper]{amsart}
\usepackage[margin=3cm]{geometry}
\usepackage{color}
\usepackage{mathrsfs}
\usepackage{mathtools}
\usepackage{graphicx,caption,subcaption} 

\usepackage[pdftex,hyperfootnotes=false,colorlinks=true,linkcolor=blue,
citecolor=purple,filecolor=magenta,urlcolor=cyan]{hyperref}

\usepackage{amsmath,amssymb,amscd,amsthm,amsfonts,anyfontsize,dsfont,%
enumerate,layout,lipsum,lpic,mathrsfs,mdwlist, pigpen,stmaryrd,thmtools,tikz,txfonts,xspace}
\usepackage[all]{xy}
\usepackage[oldstyle]{libertine}%
\usepackage[T1]{fontenc}
\usepackage[utf8]{inputenc}
\usepackage{csquotes}
\makeatletter
\renewcommand*\libertine@figurestyle{LF}
\makeatother
\usepackage[libertine,libaltvw,liby]{newtxmath}
\makeatletter
\renewcommand*\libertine@figurestyle{OsF}
\makeatother
\usepackage{fancyhdr}
\usepackage{tikz-cd}
\usepackage[noabbrev,capitalise]{cleveref}

\usepackage{autonum}
\newtheorem{theorem}{Theorem}[section]
\newtheorem{lemma}[theorem]{Lemma}
\newtheorem{proposition}[theorem]{Proposition}
\newtheorem{corollary}[theorem]{Corollary}

\theoremstyle{definition}
\newtheorem{definition}[theorem]{Definition}
\newtheorem{remark}[theorem]{Remark}

\newtheorem{notation}[theorem]{Notation}

\newcommand{\cor}[1]{\bigg{\langle} \,  #1 \, \bigg{\rangle}}

\newcommand{\diff}{\;\textrm{d}}

\newcommand{\RM}[1]{%
    \textup{\uppercase\expandafter{\romannumeral#1}}%
  }

\DeclareMathOperator\Aut{Aut}

\def\Res{\mathrm{Res}}

\def\Q{\mathbb{Q}}

\def\C{\mathbb{C}}

\def\J{\mathrm{J}}

\def\CP1{\mathbb{C}\mathrm{P}^1}

\makeatletter
\def\paragraph{\@startsection{paragraph}{4}%
  \z@\z@{-\fontdimen2\font}%
  {\normalfont\itshape}}
\makeatother

\DeclareRobustCommand{\stirling}{\genfrac\{\}{0pt}{}}
\DeclareRobustCommand{\fstirling}{\genfrac[]{0pt}{}}

\numberwithin{equation}{section}

\title[Triply mixed coverings of arbitrary base curves]{Triply mixed coverings of arbitrary base curves: Quasimodularity, quantum curves and a mysterious topological recursions}
\author[M.~A.~Hahn]{Marvin Anas Hahn}
\address{M.~A.~Hahn: Institut für Mathematik, Goehte-Universität Frankfurt, Robert-Mayer-Str. 6-8, 60325 Frankfurt am Main}
\email{hahn@math.uni-frankfurt.de}
\author[J.~W.~M.~van~Ittersum]{Jan-Willem M. van Ittersum}
\address{J.~W.~M.~v.~Ittersum: Mathematisch Instituut, Universiteit Utrecht, Postbus 80.010, 3508 TA Utrecht, The Netherlands; 
Max-Planck-Institut f\"ur Mathematik, Vivatsgasse 7, 53111 Bonn, Germany}
\email{j.w.m.vanittersum@uu.nl}
\author[F.~Leid]{Felix Leid}
\address{F.~Leid.: Universit\"at des Saarlandes, Fachrichtung Mathematik, Postfach 151150, 66041 Saarbr\"ucken, Germany}
\email{leid@math.uni-sb.de}


\begin{document}
\begin{abstract}
Simple Hurwitz numbers are classical invariants in enumerative geometry counting branched morphisms between Riemann surfaces with fixed ramification data. In recent years, several modifications of this notion for genus $0$ base curves have appeared in the literature. Among them are so-called monotone Hurwitz numbers, which are related to Harish-Chandra--Itzykson--Zuber integral in random matrix theory and strictly monotone Hurwitz numbers which enumerate certain Grothendieck dessins d'enfants. We generalise the notion of Hurwitz numbers to interpolations between simple, monotone and strictly monotone Hurwitz numbers for arbitrary genera and any number of arbitrary but fixed ramification profiles. This yields generalisations of several results known for Hurwitz numbers. When the target surface is of genus one, we show that the generating series of these interpolated Hurwitz numbers are quasimodular forms. In the case that all ramification is simple, we refine this result by writing this series as a sum of quasimodular forms corresonding to tropical covers weighted by Gromov-Witten invariants. Moreover, we derive a quantum curve for monotone and Grothendieck dessins d'enfants Hurwitz numbers for arbitrary genera and one arbitrary but fixed ramification profile. Thus, we obtain spectral curves via the semiclassical limit as input data for the Chekhov--Eynard--Orantin (CEO) topological recursion. Astonishingly, we find that the CEO topological recursion for the genus $1$ the spectral curve of the strictly monotone Hurwitz numbers compute the monotone Hurwitz numbers in genus $0$. Thus, we give a new proof that monotone Hurwitz numbers satisfy CEO topological recursion. This points to an unknown relation between those enumerants. Finally, specializing to target surface $\mathbb{P}^1$, we find recursions for monotone and Grothendieck dessins d'enfants double Hurwitz numbers, which enables the computation of the respective Hurwitz numbers for any genera with one arbitrary but fixed ramification profile.
\end{abstract}
\maketitle
%


\section{Introduction}
Hurwitz numbers are enumerations of branched morphisms between Riemann surfaces with fixed ramification data. They were first introduced by Adolf Hurwitz in the late 19th century \cite{Hurwitzverzweigung} as coverings of the Riemann sphere. As observed by Hurwitz himself, these enumerations are closely related to the combinatorics of the symmetric group \cite{Hurwitzsymmetric}. In particular this yields an elegant interpretation of Hurwitz numbers in terms of factorisations in the symmetric group. In the last two decades Hurwitz numbers have branched out into several areas of mathematics, such as algebraic geometry, Gromov-Witten theory, algebraic topology, representation theory of the symmetric group, operator theory, integrable systems, random matrix models, tropical geometry and many more.\par
Several specifications of Hurwitz numbers with respect to the genera of the involved Riemann surfaces and the ramification profiles of the morphisms have proved to be of particular interest. A common theme is to allow a finite number of arbitrary but fixed ramification profiles and a simple ramification profile everywhere else. Hurwitz numbers obtained by such specifications are called \textit{simple Hurwitz numbers}. Among the most important ones are \textit{single} and \textit{double} Hurwitz numbers. \par
Additionally, there are variants of Hurwitz numbers obtained by counting factorisations in the symmetric group as above, but with additional conditions. So far, these have been studied for target surfaces of genus $0$. Two of the most important cases are \textit{monotone} and \textit{strictly monotone} Hurwitz numbers, the latter of which are also called \textit{Grothendieck dessins d'enfants Hurwitz numbers} \cite{GGPNmonotone,DKmonotone,KLSmonotonewedge,DMstrictly,GGPNpoly, Hahnmonodromy,HKLtriply}. Monotone Hurwitz numbers appear as coefficients in the expansion of the HCIZ integral in random matrix theory \cite{GGPNmonotone}, while strictly montone Hurwitz numbers are equivalent to counting certain Grothendieck dessins d'enfants \cite{ALSramifications, KLSmonotonewedge}.

In \cite{HKLtriply} a combinatorial interpolation between simple, monotone and strictly monotone Hurwitz numbers was introduced for genus $0$ target surfaces and two arbitrary but fixed ramification profiles. This interpolation is called \textit{triply mixed Hurwitz numbers}. In this work, we generalise the notion of triply mixed Hurwitz numbers to 
arbitrary genera and any number of arbitrary but fixed ramification profiles. We study these objects from several perspectives. This yields generalisations of various results known for simple Hurwitz numbers, which in particular specialise to the extremal cases of monotone and Grothendieck dessins d'enfant Hurwitz numbers.

\subsection{Previous results}
We summarize some of the previous results on several kinds of Hurwitz numbers, which motivated our work.

\subsubsection*{Quasimodularity}
Hurwitz numbers with a target surface of genus $1$ can be expressed in terms of so-called shifted symmetric polynomials (see \cref{sec:shiftedsymmetric}). The Bloch--Okounov theorem \cite{BO00}, initally proved in the special cases corresponding to Hurwitz numbers for which all ramification are simple \cite{Dij95,KZ95}, implies that certain generating series associated to shifted symmetric polynomials are quasimodular. This implies that the generating series of (connected) Hurwitz numbers with a target surface of genus $1$ are quasimodular, as noted in \cite{EO01}. From the discussion at the end of Section 3 in \cite{GGPNpoly}, it follows that monotone Hurwitz numbers are shifted symmetric polynomials as well.

\subsubsection*{Refined quasimodularity and tropical covers}
A common theme in tropical geometry is to express geometric enumerative problems in terms of weighted graphs. In \cite{BBBMmirror}, Hurwitz numbers with only simple ramification were related to so-called \textit{tropical covers}, i.e. piecewise linear maps between metric graphs. In particular, Hurwitz numbers were expressed as a finite sum of weighted tropical covers. This led to the conjecture (which was proved in \cite{GMNquasimodular}), that each generating series obtained by considering all covers with source curves of a fixed combinatorial type are quasimodular as well. This refines the aforementioned result of \cite{Dij95,KZ95}. A tropical interpretation of monotone and strictly monotone double Hurwitz numbers for genus $0$ target surfaces was first found in \cite{DKmonotone,Hahnmonodromy} by equipping the involved tropical covers with an additional colouring and labeling. Motivated by this work, Lewanski and the first author derived a different interpretation of monotone and strictly monotone double Hurwitz numbers for genus $0$ target surfaces, which is more natural in the sense that it does not require additional colouring and labeling. The involved covers are now weighted by $1-$point relative Gromov-Witten invariants.

\subsubsection*{Topological recursion and quantum curves}
In recent years, one of the most fruitful interactions with Hurwitz theory has been from the viewpoint of Chekhov--Eynard--Orantin (CEO) topological recursion (see \cite{EMSlaplace}). CEO topological recursion associates to a spectral curve a family of differentials, which satisfies a certain recursion. Remarkably enough, for many enumerative invariants one can find spectral curves, such that the associated differentials encode these invariants as coefficients in local expansions. One says that such an enumerative problem \textit{satisfies CEO topological recursion}. It turns out that certain kinds of Hurwitz numbers satisfy CEO topological recursion \cite{EMSlaplace,DLNorbifold,DDMmonotone,DKmonotone,Norburyquantum, DMSSspectral,DOPSstrict,KZvirasoro}. A direct consequence of CEO topological recursion is an interpretation of the enumerative problem in terms of intersection products on the moduli space of stable curves with marked points $\overline{\mathcal{M}}_{g,n}$.\par 
One approach to CEO topological recursion which has been very successful is in terms of so-called quantum curves. Given an enumerative problem, an associated quantum curves is a certain partial differential equation, which is satisfied by a generating series of the initial enumeration. This quantum curve is often time an indication for the shape of the spectral curve for which one then runs CEO topological recursion.

\subsection{Results of this paper}
First of all we define our new enumerative object, i.e. \textit{triply mixed Hurwitz numbers} for target surfaces of higher genera. We then study several specifications of these numbers from various perspectives.

\subsubsection*{Quasimodularity}
When the target surface is of genus $1$, we express triply mixed Hurwitz numbers in terms of shifted symmetric functions. This allows to prove that the generating series of triply mixed Hurwitz numbers are quasimodular forms of mixed weight. Our results are summarised in \cref{thm:genquasi}.

\subsubsection*{Refined Quasimodularity}
We use the aforementioned result of Lewanski and the first author and derive an expression of monotone and Grothendieck dessins d'enfants Hurwitz numbers for genus $1$ target surfaces and only simple ramification in terms of tropical covers. This enables us to prove that fixing the combinatorial type of the source curve of the tropical covers yields a quasimodular form. This refines the quasimodularity esult for triply mixed Hurwitz numbers analogously to the simple case. The tropical correspondence theorem is stated in \cref{thm:tropel} and the refined quasimodularity statement in \cref{thm:quasi}.

\subsubsection*{Quantum curves}
Motivated by work of Liu-Mulase-Sorkin \cite{LMSsimplequantum}, we derive a quantum curve for monotone and Grothendieck dessins d'enfants Hurwitz numbers for arbitrary genera and one arbitrary but fixed ramification profile. The result for the monotone enumerations can be found in \cref{thm:mono}.

\subsubsection*{A mysterious topological recursion}
We consider the the quantum curve for Gorthendieck dessins d'enfants Hurwitz numbers for genus $1$ base curves. This quantum curve yields a spectral curve via its semi-classical limit. We use this spectral curve as input data for the CEO topological recursion. Astonishingly, we prove in \cref{sec:myst} that the expansion of the resulting differentials yield the monotone Hurwitz numbers for genus $0$ base curves, however for a different normalisation and different spectral curve than the one in \cite{DDMmonotone}. This points towards an unknown relation between the strictly monotone numbers in genus $1$ and the monotone ones in genus $0$.

\subsubsection*{Further recursive procedures}
We prove recursions for refinements of monotone and Grothendieck dessins d'enfants double Hurwitz numbers for the target surface $\mathbb{P}^1$, which yields the Hurwitz numbers for any genera with one arbitrary but fixed ramification profile. This generalises the recursion for monotone orbifold Hurwitz numbers in \cite{DKmonotone}. The explicit recursions can be found in \cref{thm:recur}.

\subsubsection*{Structure of the paper}
In \cref{sec:pre}, we recall the relevant notions needed for our discussion. Triply mixed Hurwitz numbers are introduced in \cref{sec:triplymixed}. Quasimodularity of these Hurwitz numbers is shown in \cref{sec:quasimodularity} and refined quasimodularity is discussed in \cref{sec:refined}.
In \cref{sec:quan} we derive quantum curves for both monotone and Grothendieck dessins d'enfants Hurwitz numbers. In \cref{sec:myst}, we discuss the special case of the quantum curve for Grothendieck dessins d'enfants Hurwitz numbers with elliptic base curve and prove that CEO topological recursion for the semiclassical limit of this quantum curve computes monotone Hurwitz numbers with rational base curve. Finally, we derive the recursions for refinements of monotone Grothendieck dessins d'enfants double Hurwitz numbers in \cref{sec:refine}. We collected several examples of quasimodular generating series of triply mixed Hurwitz numbers in \cref{sec:examples}.

\subsection{Acknowledgements}
The authors thank Hannah Markwig for her careful proofreading and many helpful discussions and James Mingo for the fruitful discussions on the results of \cref{sec:myst} during the program “New Developments in Free Probability and Applications” at the CRM Montreal. We are further indebted to Gaetan Borot, Olivia Dumitrescu, Martin Moeller, Motohico Mulase and Roland Speicher for many helpful comments. We would like to thank Janko Boehm, Thomas Breuer and Dimitry Noshchenko for their help with some of the computer-aided calculations in this work. The authors M.A.H. and F.L. gratefully acknowledge partial support by DFG SFB-TRR 195 “Symbolic tool in mathematics and their applications”, project A 14 “Random matrices and Hurwitz numbers” (INST 248/238-1). Moreover, M.A.H. gratefully acknowledges support as part of the LOEWE research unit Uniformized Structures in Arithmetic and Geometry.

\section{Preliminaries}\label{sec:pre}

\subsection{Hurwitz numbers}\label{sec:Hurwitz}

In this section, we recall some of the basic notions of Hurwitz theory. We begin by defining classical Hurwitz numbers in the most general sense. We also review some of the specifications and variations on the definition of Hurwitz numbers relevant for this paper.

\subsubsection{Classical Hurwitz numbers}
Let $\lambda$ be a composition, i.e. a finite sequence of strictly positive integers. Denote by $|\lambda|$ the integer where $\lambda$ is a composition of and let $\ell(\lambda)$ be the number of parts of $\lambda$. Write $p(\lambda)$ for the ordered composition (partition) corresponding to $\lambda$. Given a permutation $\sigma\in S_d$, denote by $c(\sigma)$ the partition which corresponds to the cycle type of $\sigma$.
\begin{definition}
\label{def:hurwitznumber}
Let $g',g\ge0$ be non-negative integers, $d$ a positive integer and $\mu=(\mu^1,\dots,\mu^n)$ a tuple of compositions of $d$. In case 
\[2g'-2=d\cdot(2g-2)+\sum_{j=1}^n|\mu^j|-\ell(\mu^j),\]
 we call $(\sigma_1,\dots,\sigma_n,\alpha_1,\beta_1,\dots,\alpha_g,\beta_g)$ a \emph{factorization of type} $(g,g',d,\mu)$ if the following conditions are satisfied:
\begin{enumerate}
\item $\sigma_i,\alpha_i,\beta_i\in S_d$,
\item $\sigma_1\cdots\sigma_n=[\alpha_1,\beta_1]\cdots[\alpha_{g},\beta_{g}]$,
\item $c(\sigma_i)=p(\mu^i)$.
\end{enumerate}
If additionally we have
\begin{itemize}
\item[(4)] $\langle\sigma_1,\dots,\sigma_n,\alpha_1,\beta_1,\dots,\alpha_g,\beta_g\rangle$ acts transitively on the set $\{1,2,\ldots, d\}$
\end{itemize}
we call the factorization \emph{connected}.  Denote by $\mathcal{F}^\bullet(g,g',d,\mu)$ and $\mathcal{F}(g,g',d,\mu)$ the factorizations respectively connected factorizations of type $(g,g',d,\mu)$. The \emph{Hurwitz numbers} and \emph{connected Hurwitz numbers} are defined by
\[h^{\bullet,g}_{g'}(\mu^1,\dots,\mu^n) = \frac{1}{d!}|\mathcal{F}^\bullet(g,g',d,\mu)| \quad \text{respectively} \quad h^{g}_{g'}(\mu^1,\dots,\mu^n) = \frac{1}{d!}|\mathcal{F}(g,g',d,\mu)|.\]
\end{definition}

\begin{remark}
Historically, Hurwitz numbers were first defined as an enumerative problem counting ramified morphisms between Riemann surfaces: For a fixed compact Riemann surface $S$ of genus $g$, Hurwitz numbers count holomorphic maps $\pi:S'\to S$ (up to isomorphism), where $S'$ is a compact Riemann surface of genus $g'$, such that
\begin{itemize}
\item $\pi$ has ramification profile $\mu^1,\dots,\mu^n$ over $n$ arbitrary, but fixed points on $S$,
\item each map is weighted by $\frac{1}{|\mathrm{Aut}(\pi)|}$.
\end{itemize}
The connection to our definition, which is due to Hurwitz, is made by considering the monodromy representations for the holomorphic maps involved (see \cite{Hurwitzverzweigung,Hurwitzsymmetric}).
\end{remark}

Sometimes, we count Hurwitz numbers with labels in order to distinguish cycles of the same length. The following definition makes this more precise. We note that geometrically, this corresponds to labelling the pre-images of all branch points. 

\begin{definition}
A \textit{labeled factorization} is a factorization as defined by \cref{def:hurwitznumber} together with a labelling:
\begin{itemize}
\item[(5)]the disjoint cycles of $\sigma_i$ for all $i$ are labeled with labels $1,\ldots, \ell(\mu^i)$ such that a cycle of $\sigma_i$ with label $j$ has length $(\mu^i)_j$.
\end{itemize}
We use an arrow to indicate that we are considering labeled factorizations and Hurwitz numbers, just as we are using a dot to denote non-necessarily connected factorizations and Hurwitz numbers.
\end{definition}
\begin{remark}The labeled Hurwitz numbers equal the ordinary Hurwitz numbers up to a constant:
\[\vec{h}^{\bullet,g}_{g'}(\mu^1,\dots,\mu^n) = \Aut\mu\cdot {h}^{\bullet,g}_{g'}(\mu^1,\dots,\mu^n)\]
with $\Aut\mu=\prod_{m=1}^\infty \prod_{i=1}^n r_m(\mu^i)!$ and $r_m(\mu^i)$ denotes the number of parts equal to $m$ in the partition $\mu^i$. 
\end{remark}

\subsubsection{Special instances of Hurwitz numbers}
There are several special cases and variations of the above general definition of classical Hurwitz numbers. A general theme is to allow a few complicated partitions and force almost all partitions to be \textit{simple}, i.e. equal to the partition $(2,1,\dots,1)$. In the focus of this paper are the following two special cases of Hurwitz numbers.

\begin{definition}\label{def:simplehurwitz} Let $T:=(2,1,\ldots, 1)$ be the simple partition of size $d$. In the same setting as in \cref{def:hurwitznumber}, we consider two cases:
\begin{enumerate}
\item If $\mu^2,\dots,\mu^n=T$, we call the resulting Hurwitz number \textit{single base $g$ Hurwitz number} and denote it by $H^{\bullet,g}_{g'}(\mu^1_1,\ldots, \mu^1_{\ell(\mu^1)})\coloneqq h^{\bullet,g}_{g'}(\mu^1,T,\dots,T)$. 
Note that the number $r$ of simple partitions is then given by $r=2g'-2+\ell(\mu^1)-d(2g-1)$.
\item If $g=0$ and $\mu^3,\dots,\mu^n=T$, we call the resulting number a \textit{double Hurwitz number} and denote it by $h^{\bullet}_{g'}(\mu^1,\mu^2)\coloneqq  h^{\bullet,0}_{g'}(\mu^1,\mu^2,T,\dots,T)$. 
Note that the number $r$ of simple partitions is then given by $r=2g-2+\ell(\mu^1)+\ell(\mu^2)$.
\end{enumerate}
\end{definition}

\subsubsection{Monotone and strictly monotone Hurwitz numbers}
There are several variants of Hurwitz numbers of which two are relevant for this work: monotone and strictly monotone Hurwitz numbers. In the following, we define those numbers.

\begin{definition}\label{MonotoneBFactorization}
Let $\mu=(\mu^1,\dots,\mu^{k+2})$ be a tuple of compositions of $d$. We call a (connected/labeled) factorization of type $(0,g',d,\mu)$, where $\mu^3=\dots=\mu^{k+2}=(2,1\dots)$ a (connected/labeled) \emph{monotone factorization of type} $(g',\mu,\nu)$ if it satisfies the additional property 
\begin{itemize}
\item[(6)] The transpositions $\sigma_i=(s_i\,t_i)$ with $s_i<t_i$ satisfy $t_i\le t_{i+1}$ for $i=3,\dots,k+1$.
\end{itemize}
We then define the monotone Hurwitz number $h^{\bullet}_{\le,g'}(\mu^1,\mu^2)$ as the product of $\frac{1}{d!}$ and the number of monotone factorisations of type $(g',\mu,\nu)$ and similarly define connected/labeled monotone double Hurwitz numbers. 
We define \emph{strictly monotone double Hurwitz numbers}, also called \textit{Grothendieck dessins d'enfant Hurwitz numbers}, by changing the monotonicity condition to a strict one
\begin{itemize}
\item[(6')] The transpositions $\sigma_i=(s_i\,t_i)$ with $s_i<t_i$ satisfy $t_i< t_{i+1}$ for $i=3,\dots,k+1$.
\end{itemize}
We denote the strictly monotone double Hurwitz number by $h^{\bullet}_{<,g'}(\mu^1,\mu^2)$.
\end{definition}

We note that \cref{def:trip} generalizes (strictly) monotone Hurwitz numbers to target surfaces of higher genera.

%

\subsection{Quantum curves}
In \cite{LMSsimplequantum}, the connected single base $g$ Hurwitz numbers were studied with a view towards topological recursion. The main results are summarized in \cref{thm:lms} below. 

Denote by $\mathbb{Z}_+$ the set of positive integers and an element $\nu\in\mathbb{Z}_+^n$ by $\nu=(\nu_1,\dots,\nu_n)$. Define the discrete Laplace transform of the single base $g$ Hurwitz numbers by 
\begin{align}
F^g_{g'}(x_1,\dots,x_n)=&\sum_{\nu\in\mathbb{Z}^n_{+}}\vec{H}_{g'}^g(\nu_1,\dots,\nu_n)\prod_{i=1}^ne^{-w_i\nu_i} 
\quad \quad (x_i=e^{-w_i}),
\end{align}
which is usually referred to as the \textit{free energy}. We further define the so-called \textit{partition function} by
\begin{equation}
Z^g(x,\hbar)=\mathrm{exp}\left(\sum_{g'=1}^{\infty}\sum_{n=1}^{\infty}\frac{1}{n!}\hbar^{2g'-2+n}F_{g'}^g(x,\dots,x)\right).
\end{equation}

The following theorem was proved in \cite{LMSsimplequantum}.

\begin{theorem}[{\cite[Theorem 1.1]{LMSsimplequantum}}]
\label{thm:lms}
For $2g-2+\ell(\mu)>0$, the free energies $F^g_{g'}(x_1,\dots,x_n)$ satisfy a cut-and-join type partial differential equation.\par
The single base $g$ Hurwitz partition function satisfies a quantum curve-like infinite order differential equation
\begin{equation}
\hbar x\frac{d}{dx}\left[1-\hbar^{1-\chi}xe^{\hbar x\frac{d}{dx}}\left(\frac{d}{dx}x\right)^{1-\chi}\right]Z^g(x,\hbar)=0,
\end{equation}
where $\chi=2-2g$ (which is the Euler characteristic of the base curve of the Hurwitz problem). If we introduce
\begin{equation}
y=\hbar x\frac{d}{dx}
\end{equation}
and regard it as a commuting variable, then the total symbol of the above operator produces the following equation
\begin{equation}
x=y^{\chi-1}e^{-y},
\end{equation}
which is commonly refered to as a \textit{Lambert curve}.
\end{theorem}

\begin{remark}
A cut-and-join type partial differential equation is a certain kind of partial differential equation, which reflects the combinatorics of multiplying elements in the symmetric group. In the case of simple Hurwitz numbers, these combinatorics entail the fact that left-multiplication by a transposition either joins two cycles to one or cuts a single cycle into two.
\end{remark}

\subsection{Stirling numbers}
We now define Stirling numbers of the first and second kind.

\begin{definition}
For $n,k\in\mathbb{N}$, we define \textit{Stirling numbers of the first kind} by the recurrence relation
\begin{equation}
\fstirling{n+1}{k}=n\fstirling{n}{k}+\fstirling{n}{k-1}\,\,\mathrm{for}\,\,k>0;\quad \fstirling{0}{0}=1\,\,\mathrm{and}\,\,\fstirling{n}{0}=\fstirling{0}{n}=0\,\,\mathrm{for}\,\,n>0
\end{equation}
and \textit{Stirling numbers of the second kind} by the recurrence relation
\begin{equation}
\stirling{n+1}{k}=k\stirling{n}{k}+\stirling{n}{k-1}\,\,\mathrm{for}\,\,k>0;\quad\stirling{0}{0}=1\,\,\mathrm{and}\,\,\stirling{n}{0}=\stirling{0}{n}=0\,\,\mathrm{for}\,\,n>0.
\end{equation}
\end{definition}

\begin{remark}
Recall the generating functions of Stirling numbers
\begin{equation}\label{eq:stirling}
\sum_{k=0}^{n}\fstirling{n}{k}\frac{z^{n-k}}{n!}
=\prod_{r=1}^{n-1} (1-rz)\quad\mathrm{and}\quad\sum_{n=k}^{\infty}\stirling{n}{k}x^{n-k}=\prod_{r=1}^k\frac{1}{1-rx}.	
\end{equation}
\end{remark}

\subsection{Shifted symmetric functions and quasimodular forms}
\subsubsection{Central characters}\label{sec:centralcharacter}
Let $\mathcal{Z}_n$ be the center of the group algebra $\C[S_n]$ of the symmetric group. Given a partition $\lambda$, denote by $|\lambda|$ the integer where $\lambda$ is a partition of and by $C_\lambda$ the sum of all elements in $S_{|\lambda|}$ of cycle type $\lambda$. Let $(\rho_\lambda,V_\lambda)$ be the irreducible representation of $S_{|\lambda|}$ corresponding to $\lambda$. Note that a basis of $\mathcal{Z}_n$ is given by $C_\lambda$ for all partitions $\lambda$ of $n$. Every element $z\in \mathcal{Z}_n$ defines by Schur's lemma a constant, called the \textit{central character} $\omega^\lambda(z)$, by which $\rho_\lambda(z)$ acts on $V_\lambda$. For example, the central character of the conjugacy class sums $C_\nu$ is given by

\begin{equation}\label{def:fnu}
f_\nu(\lambda):=\omega^\lambda(C_\nu)  = |C_\nu|\frac{\chi^\lambda(\nu)}{\dim\lambda},
\end{equation}
where (by abuse of notation) $|C_\nu|$ denotes the size of the conjugacy class of elements of cycle type $\nu$ in $S_{|\nu|}$, $\chi^\lambda$ denotes the character of $\rho_\lambda$ and $\dim\lambda$ the dimension of $\rho_\lambda$. We extend \cref{def:fnu} to the case when $|\nu|\neq |\lambda|$ by
\begin{equation}
f_\nu(\lambda):=\binom{|\lambda|}{|\nu|}|C_\nu|\frac{\chi^\lambda(\nu)}{\dim\lambda}.
\end{equation}
Observe that if $\nu$ is a partition without parts equal to $1$ the quantity $\binom{|\lambda|}{|\nu|}|C_\nu|$ equals the size of the conjugacy class of cycle type $\nu$ in $S_{|\lambda|}$ instead of in $S_{|\nu|}$. 
For a tuple of partitions $\mu=(\mu^1,\dots,\mu^n)$ let $f_\mu = \prod_{i=1}^n f_{\mu^i}$. Hurwitz numbers can be expressed in terms of central characters, for example classical Hurwitz numbers of torus coverings satisfy
\begin{equation}\label{eq:hurwitzss} h^{\bullet,1}_{g'}(\mu)  = \sum_{|\lambda|=d} f_\mu(\lambda). 
\end{equation}
with $d$ --- as in the definition of Hurwitz numbers --- implicitly given by the size of the partitions where $\mu$ consists of. 
In \cref{prop:hurwitzcentralcharacter} we show that triply mixed Hurwitz numbers are a polynomial in central characters. 

Let $\Xi_d=(J_1,J_2,\ldots,J_d,0,0,\ldots)$ be the sequence of Jucys-Murphy elements given by $J_k = \sum_{i=1}^{k-1} (i\,k)$. Although the Jucys-Murphy elements do not belong to $\mathcal{Z}_n$, symmetric polynomials in $\Xi_d$ are elements of $\mathcal{Z}_n$. More precisely, every element of $\mathcal{Z}_n$ can be written as a symmetric polynomial in $\Xi_d$ \cite{jucys1974symmetric,murphy1981new}. Remarkably,  the central character of a symmetric polynomial $f$ evaluated at $\Xi_d$ simply equals
 $$\omega^\lambda(f(\Xi_d))=f(\mathrm{cont}_\lambda),$$
 where $\mathrm{cont}_\lambda$ denotes the sequence of all contents of $\lambda$. 
 
The commutator sum $\mathfrak{K} = \sum_{\alpha,\beta\in S_d} [\beta,\alpha]$ also is an element of $\mathcal{Z}_n$. Its central character equals
\begin{align}
\omega_\lambda(\mathfrak{K}) 
&=\left(\frac{d!}{\dim\lambda}\right)^2. 
\end{align}
 
 \subsubsection{Central characters as shifted symmetric polynomials}\label{sec:shiftedsymmetric}
Central characters of $z\in \mathcal{Z}_n$ are examples of \textit{shifted symmetric polynomials}, introduced by Okounkov and Olshanski \cite{OO97}. A rational polynomial in $m$ variables $x_1,\ldots, x_m$ is called \emph{shifted symmetric} if it is invariant under the action of all $\sigma\in\mathfrak{S}_m$ given by $x_i\mapsto x_{\sigma(i)}+i-\sigma(i)$ (or more symmetrically $x_i-i\mapsto x_{\sigma(i)}-\sigma(i)$). Denote by $\Lambda^*(m)$ the space of shifted symmetric polynomials in $m$ variables, which is filtered by the degree of the polynomials. We have forgetful maps $\Lambda^*(m)\to \Lambda^*({m-1})$ given by $x_m\mapsto 0$, so that we can define the space of shifted symmetric polynomials $\Lambda^*$ as $\displaystyle\varprojlim_m \Lambda^*(m)$ in the category of filtered algebras. As $\omega^\lambda(z)$ is a shifted symmetric polynomial, it can be expressed as a symmetric polynomial in $\lambda_1-1,\lambda_2-2,\ldots$. For example,
\[f_{()}(\lambda)=1,\quad  f_{(1)} = |\lambda| = \sum_{i=1}^\infty ((\lambda_i-i)+i), \quad f_{(2)}(\lambda) = \tfrac{1}{2}\sum_{i=1}^\infty (\lambda_i-i+\tfrac{1}{2})^2-(-i+\tfrac{1}{2})^2.\]
More precisely, the algebra of shifted symmetric functions $\Lambda^*$ has a basis $f_\nu$ where $\nu$ ranges over all partitions. 

\subsubsection{The Bloch--Okounov theorem}
A distinguished generating set for $\Lambda^*$ is given by the renormalized shifted symmetric power sums
\[ Q_0(\lambda)=1, \quad \quad Q_k(\lambda) = c_k+\frac{1}{(k-1)!}\sum_{i=1}^\infty (\lambda_i-i+\tfrac{1}{2})^{k-1}-(-i+\tfrac{1}{2})^{k-1} \quad \quad (k\geq 1) \]
with $c_k$ defined by
\[\label{eq:ci} \varsigma(z) := \frac{1}{2\sinh(z/2)} =: \sum_{k=1}^\infty c_k z^{k-1}.\]
Define a weight grading on $\Lambda^*$ by assigning to $Q_k$ weight $k$. 

This weight grading corresponds to the weight of quasimodular forms under the Bloch--Okounov theorem, as follows. The graded algebra of quasimodular forms is given by $\widetilde{M}=\Q[P,Q,R]$, where $P=-24G_2,Q=240G_4,R=-504G_6$ are Ramanujan's notation for the Eisenstein series
$$ G_k(\tau) = -\frac{B_k}{2k}+\sum_{r=1}^\infty\sum_{m=1}^\infty m^{k-1} q^{mr}, \quad \quad (B_k = k\text{th Bernoulli number and } q=e^{2\pi i \tau})$$
of weight $k$. 
Given a function $f$ on partitions, introduce the $q$-bracket of $f$, given by
$$\langle f \rangle_q = \frac{\sum_{\lambda\in \mathscr{P}} f(\lambda) q^{|\lambda|}}{\sum_{\lambda\in \mathscr{P}}q^{|\lambda|}}\in \C[[q]].$$
The denominator $\sum_{\lambda\in \mathscr{P}}q^{|\lambda|}$ equals $q^{1/24}\eta(\tau)^{-1}$ with $\eta(\tau)$ the Dedekind eta function. Then, by the celebrated Bloch--Okounov theorem the $q$-bracket $\langle f\rangle_q$ of a shifted symmetric polynomial $f$ is the $q$-expansion of a quasimodular form \cite[Theorem 0.5]{BO00}. Moreover, if $f$ has weight $k$ as defined above, then $\langle f\rangle_q$ is quasimodular of the same weight $k$. 

\subsubsection{Weights of central characters}\label{sec:weights}
We give the (mixed) weight of the central characters defined in \cref{sec:centralcharacter} in terms of the weight grading on $\Lambda^*$. Define completion coefficients $q_{k,\nu}$ with $k\geq 2$ and $\nu$ a partition by 
\[Q_k = \sum_{\nu} q_{k,\nu} f_\nu.\]
The Gromov-Witten--Hurwitz correspondence provides the following formula for these coefficients: 
\begin{proposition}[{\cite[Proposition 3.2]{OP06}}]
The completion coefficients $q_{k,\nu}$ 
satisfy
\[\sum_{k=1}^\infty q_{k+1,\nu}z^{k} = \frac{1}{|\nu|!}(e^{z/2}-e^{-z/2})^{|\nu|-1}\prod_{i=1}^{\ell(\nu)} (e^{\nu_iz/2}-e^{-\nu_i z/2}).\]
\end{proposition}
In particular, $q_{k,\nu}=0$ if $|\nu|+\ell(\nu)>k$. This implies that $f_{\nu}$ is of (mixed) degree at most $|\nu|+\ell(\nu)$, which was already proved in \cite[Theorem 5]{KO94}.

For the weights of symmetric polynomials evaluated at the contents of a partition we have the following result. Let $h_n$ and $e_n$ be the complete homogeneous symmetric polynomial respectively elementary symmetric polynomial of degree $n$. 
\begin{proposition}\label{prop:2d} Let $d\geq 1 $. The top-weight part of $h_d(\mathrm{cont}_\lambda)$ and $e_d(\mathrm{cont}_\lambda)$ is given by $\frac{1}{2^{d-1}}Q_3(\lambda)^{d}$.
\end{proposition}
\begin{proof}
Let $(x)_k=x(x-1)\cdots(x-k+1)$. Then, for the symmetric polynomial $\underline{p}_k(x_1,\ldots,x_n):=\sum_{i=1}^n(x_i)_k$ it is known that \cite[Theorem 4]{KO94}
\begin{equation}\label{eq:ko}k\underline{p}_{k-1}(\mathrm{cont}_\lambda) = \sum_{i=1}^\infty (\lambda_i-i+1)_k-(-i+1)_k \quad \quad (k\geq 1)\end{equation}
where the right-hand side is a shifted symmetric polynomial of mixed degree $\leq k+1$. 

 Given a symmetric polynomial $g$ of degree $d$ and with constant term equal to zero, one can write $g$ as a polynomial in the $\underline{p}_k$ for $k\geq 1$. Assign to $\underline{p}_k$ weight $k+2$ in accordance with \cref{eq:ko}, i.e. $\underline{p}_k(\mathrm{cont}_\lambda)$ is a shifted symmetric polynomial of mixed weight $k+2$. Observe that the monomial $\underline{p}_1^d$ is the unique monomial of degree $d$ and weight at least (precisely) $3d$. Hence, the central character $\omega^\lambda(f(\Xi_d))$ of $g$ is a shifted symmetric function of weight at most $3d$ with top-degree part up to a multiplicative constant equal to $Q_3(\lambda)^d$. Specializing $g$ to $h_d$ and $e_d$ the result follows by observing that in this case the coefficient of $\underline{p}_1^d$ equals $\frac{1}{2^{d-1}}$. 
\end{proof}

\subsection{Gromov-Witten invariants with target $\mathbb{P}^1$}
In this section, we introduce the basic notions of Gromov-Witten theory needed for this work. For a more concise introduction in the context of tropical geometry, see e.g. \cite{CJMRgraphical}. For a more general introduction to the topic, we recommend \cite{Vakilgromov}.\par
We denote by $\overline{M}_{g,n}(\mathbb{P}^1,d)$ the moduli space of stable maps with $n$ marked points, which a Deligne-Mumford stack of virtual dimension $2g-2+2d+n$. It consists of tuples $(X,x_1,\dots,x_n,f)$, such that $X$ is a connected, projective curve of genus $g$ with at worst nodal singularities, $x_1,\dots,x_n$ are non-singular points on $X$ and $f:X\to\mathbb{P}^1$ is a function with $f_{\ast}([X])=d[\mathbb{P}^1]$. Moreover, $f$ may only have a finite automorphism group (respecting markings and singularities). In order to define enumerative invariants, we introduce
\begin{itemize}
\item The $i-$th evaluation morphism is the map $ev_i:\overline{M}_{g,n}(\mathbb{P}^1,d)\to\mathbb{P}^1$ by mapping the tuple $(X,x_1,\dots,x_n,f)$ to $x_i$.
\item The $i-$th cotangent line bundle $\mathbb{L}_i\to\overline{M}_{g,n}(\mathbb{P}^1,d)$ is obtained by identifying the fiber of each point with the cotangent space $\mathbb{T}^*_{x_i}(X)$. The first chern class of $i-$th cotangent line bundle is called a psi class, which we denote by $\psi_i=c_1(\mathbb{L}_i)$.
\end{itemize}

This yields the following definition.

\begin{definition}
Fix $g,n,d$ and let $k_1,\dots,k_n$ be non-negative integers, such that $k_1+\dots+k_n=2g+2d-2$. Then, the stationary Gromov-Witten invariant is defined by
\begin{equation}
\langle\tau_{k_1}(pt)\cdots\tau_{k_n}(pt)\rangle_{g,n}^{\mathbb{P}^1}=\int_{[\overline{M}_{g,n}(\mathbb{P}^1)]^{vir}}\prod ev_i^*(pt)\psi_i^{k_i},
\end{equation}
where $pt$ denotes a point on $\mathbb{P}^1$.
\end{definition}

\begin{remark}
Analogously, we can define Gromov-Witten invariants for more general target curves $Y$
\begin{equation}
\langle\tau_{k_1}(pt)\cdots\tau_{k_n}(pt)\rangle_{g,n}^{Y}.
\end{equation}
The following identity, which should be compared with \cref{eq:hurwitzss} in the previous discussion on shifted symmetric functions, was proved in \cite{OP06} for elliptic curves $E$:
\begin{equation}
\langle\tau_{k_1}(pt)\cdots\tau_{k_n}(pt)\rangle_{g,n}^{E,d}=  \sum_{|\lambda|=d} \prod_{i=1}^n Q_{k_i+2}(\lambda).
\end{equation}
\end{remark}

Similarly, we consider the moduli space of relative stable maps $\overline{M}_{g,n}(\mathbb{P}^1,\nu,\mu,d)$ relative to two partitions $\mu,\nu$ of $d$ and define the relative Gromov-Witten invariants by

\begin{equation}
\langle\nu\mid\tau_{k_1}(pt)\cdots\tau_{k_n}(pt)\mid\mu\rangle_{g,n}^{\mathbb{P}^1}=\int_{[\overline{M}_{g,n}(\mathbb{P}^1,\nu,\mu,d)]^{vir}}\prod ev_i^*(pt)\psi_i^{k_i}.
\end{equation}

We note that in the following, we add subscripts "$\circ$" and "$\bullet$", which correspond to \textit{connected} or \textit{not necessarily connected} (for simplicity also called \textit{disconnected} Gromov-Witten invariants, which in turn correspond to considering connected or disconnected stable maps.

\subsection{Tropical covers and monotone/Grothendieck dessins d'enfants Hurwitz numbers}
A detailed introduction to tropical covers can be found in \cite{ABBRharmonic}. We note that all graphs considered may contain half-edges.

\begin{definition}
\label{def:abstrop}
An \textit{abstract tropical curve} is a connected
metric graph $\Gamma$, together with a function associating a genus $g(v)$ to each vertex $v$. Let $V(\Gamma)$ be the set of its vertices. Let $E(\Gamma)$ be the set of its internal edges, which we require to be bounded and let $E'(\Gamma)$ its set of all edges, respectively. The set of half-edges, which we call ends is therefore $E'(\Gamma) \setminus E(\Gamma)$, and all ends are considered to have infinite length. The genus of an abstract tropical curve $\Gamma$ is
 $ g(\Gamma)\coloneqq h^1(\Gamma) + \sum_{v \in V(\Gamma)} g(v)$,
where $h^1(\Gamma)$ is the first Betti number of the underlying graph.
An \textit{isomorphism} of a tropical curve is an automorphism of the underlying graph that respects edges' lengths and vertices' genera.
The \textit{combinatorial type} of a tropical curve is the graph obtained by disregarding its metric structure.
\end{definition}

\begin{remark}
An important tropical curve for this work is the so-called \textit{tropical projective line}, defined as $\mathbb{R}$ with finitely many decorated points. The points are the vertices, the intervals between the vertices (and from the extremal vertices to $\pm\infty$) are the edges and the length of the edges are the lengths of the intervals. We denote the tropical projective line by $\mathbb{P}^1_{trop}$ for any choice of points.
\end{remark}

We now define tropical covers. For an illustration, see e.g \cite[Figure 1]{MR3821170}.

\begin{definition}
\label{def:tropmorph}
A \emph{tropical cover} is a surjective harmonic map $\pi:\Gamma_1\to\Gamma_2$ between abstract tropical curves as in \cite[section 2]{ABBRharmonic}, i.e.:
\begin{itemize}
\item[\textit{i).}] 
$\pi(V(\Gamma_1))\subset V(\Gamma_2)$.
\item[\textit{ii).}]
 $\pi^{-1}(E'(\Gamma_2))\subset E'(\Gamma_1)$.
\item[\textit{iii).}]For each edge $e\in E'(\Gamma_i)$, denote by $l(e)$ its length. Interpreting $e\in E'(\Gamma_1),\pi(e)\in E'(\Gamma_2)$ as intervals $[0,l(e)]$ and $[0,l(\pi(e))]$, we require $\pi$ restricted to $e$ to be a linear map of slope $\omega_e\in\mathbb{Z}_{\ge0}$, that is $\pi:[0,l(e)]\to[0,l(\pi(e))]$ is given by $\pi(t)=\omega_e\cdot t$. We call $\omega_e$ the \textit{weight} of $e$. If $\pi(e)$ is a vertex, we have $\omega_e=0$.
\item[\textit{iv).}] For a vertex $v\in\Gamma_1$, let $v'=\pi(v)$. We choose an edge $e'$ adjacent to $v'$. We define the local degree at $v$ as
\begin{equation}
d_v=\sum_{\substack{e\in\Gamma_1\\\pi(e)=e'}}\omega_e.
\end{equation}
We require $d_v$ to be independent of the choice of edge $e'$ adjacent to $v'$. We call this fact the \textit{balancing} or \textit{harmonicity condition}.
\end{itemize}

We furthermore introduce the following notions:
\begin{itemize}
\item[\textit{i).}] The \textit{degree} of a tropical cover $\pi$ is the sum over all local degrees of pre-images of any point in $\Gamma_2$. Due to the harmonicity condition, this number is independent of the point in $\Gamma_2$.
\item[\textit{ii).}] For any end $e\in \Gamma_2$, we define $\mu_e$ as the partition of weights of the ends of $\Gamma_1$ mapping to $e$. We call $\mu_e$ the \textit{ramification profile} above $e$.
\end{itemize}

We call two tropical covers $\pi:\Gamma_1\to\Gamma_2$ and $\pi':\Gamma_1'\to\Gamma_2$ isomorphic, if there exists an isomorphism of graph $f:\Gamma_1\to\Gamma_1'$ respecting labels and weights, such that $\pi=\pi'\circ f$. We denote the automorphism group of a tropical cover $\pi$ by $\mathrm{Aut}(\pi)$.
\end{definition}

\begin{theorem}[\cite{HLmonotone}]
\label{thm:trop}
Let $g$ be a non-negative integer, and $\mu,\nu$ partitions of the same size $d>0$. 

\begin{align}
h_{g; \mu, \nu}^{\leq, \bullet}&=\sum_{\lambda\vdash b}
\sum_{\pi \in \Gamma( \mathbb{P}^1_{\text{trop}}, g; \mu, \nu,\lambda)}\frac{1}{|\mathrm{Aut}(\pi)|}\frac{1}{\ell(\lambda)!}\prod_{v \in V(\Gamma)} m_v \prod_{e \in E(\Gamma)} \omega_e
\\
h_{g; \mu, \nu}^{<, \bullet}&=\sum_{\lambda\vdash b}
\sum_{\pi \in \Gamma( \mathbb{P}^1_{\text{trop}}, g; \mu, \nu,\lambda)}\frac{1}{|\mathrm{Aut}(\pi)|}\frac{1}{\ell(\lambda)!}\prod_{v \in V(\Gamma)} (-1)^{1 + \mathrm{val}(v)} m_v \prod_{e \in E(\Gamma)} \omega_e
\end{align}
where $\Gamma( \mathbb{P}^1_{\text{trop}}, g; \mu, \nu,\lambda)$ is the set of tropical covers 
$
\pi: \Gamma \longrightarrow \mathbb{P}^1_{trop} = \mathbb{R}
$
with $b=2g-2+\ell(\mu)+\ell(\nu)$ points $p_1,\dots,p_b$ fixed on the codomain $\mathbb{P}^1_{trop}$, such that
\begin{itemize}
\item[\textit{i).}] The unbounded left (resp. right) pointing ends of $\Gamma$ have weights given by the partition $\mu$ (resp. $\nu$).
\item[\textit{ii).}] The graph $\Gamma$ has $l:=\ell(\lambda)$ vertices. Let $V(\Gamma) = \{v_1, \dots, v_l\}$ be the set of its vertices. Then $\pi(v_i)=p_i$. Moreover, let $w_i  = \mathrm{val}(v_i)$ be the corresponding valences (degrees).
\item[\textit{iii).}] We assign an integer $g(v_i)$ as the genus to $v_i$ and the following condition holds true

\begin{equation}
h^1(\Gamma) + \sum_{i=1}^l  g(v_i) = g.
\end{equation}
\item[\textit{iv).}] We have $\lambda_i=\mathrm{val}(v_i)+2g(v_i)-2$.
\item[\textit{v).}] For each vertex $v_i$, let $\textbf{x}^{+}$ (resp. $ \textbf{x}^-$) be the tuple of weights of those edges adjacent to $v_i$, which map to the right-hand (resp. left-hand) of $p_i$. The multiplicity $m_{v_i}$ of $v_i$ is defined to be
\begin{align}
m_{v_i} = &(\lambda_i-1)!|\mathrm{Aut}(\textbf{x}^{+})||\mathrm{Aut}(\textbf{x}^{-})|\\
&\sum_{g_1^i+g_2^i=g(v_i)}\cor{ \!\! \tau_{2g^i_2 - 2}(\omega) \!\! }_{g^i_2}^{\mathbb{P}^1, \circ} \cor{ \!\!\textbf{x}^+,   \tau_{2g^i_1 - 2 + \ell(\textbf{x}^+) + \ell(\textbf{x}^-)}(\omega) , \textbf{x}^-\!\!  }^{\mathbb{P}^1,\circ}_{g^i_1}
\end{align}
\end{itemize}
\end{theorem}

\begin{remark}
The involved Gromov-Witten invariants can be computed using 
the functions $\varsigma(z) = 2\sinh(z/2) = e^{z/2} - e^{-z/2}$ and $\mathcal{S}(z)=\frac{\varsigma(z)}{z}$:
\begin{itemize}
\item 
Recall the definition of the constants $c_i$ in \cref{eq:ci}. It is well known that
\begin{equation}
\left \langle \tau_{2l-2}(\omega) \right \rangle_{l,1}^{\mathbb{P}^1}=c_{2l}.
\end{equation}
\item It was proved in \cite{OP06} that
\begin{equation}
\left\langle\textbf{x}^+,   \tau_{2g - 2 + \ell(\textbf{x}^+) + \ell(\textbf{x}^-)} , \textbf{x}^-  \right\rangle^{\mathbb{P}^1, \circ}_g
=
\frac{1}{|\Aut (\textbf{x}^+)||\Aut (\textbf{x}^-)|}[z^{2g}] \frac{\prod_{\textbf{x}^+_i} \mathcal{S}(\textbf{x}_i z)\prod_{\textbf{x}^-_i} \mathcal{S}(\textbf{x}_i z)}{\mathcal{S}(z)}.
\end{equation}

\end{itemize}
\end{remark}

\section{Triply mixed Hurwitz numbers}\label{sec:triplymixed}
We now introduce triply mixed Hurwitz numbers. Call two compositions/partitions $\mu$ and $\mu'$ equivalent if they only differ by $1-$entries. Further, for a fixed positive integer $d$, a fixed partition $\mu$ and $\sigma\in S_d$ write $\mathcal{C}(\sigma)=\mu$ if the cycle type of $\sigma$ is equivalent to $\mu$.

\begin{definition}
\label{def:trip}
Let $g',g\ge0$ be non-negative integers, $d$ a positive integer and $\mu=(\mu^1,\dots,\mu^n)$ a tuple of ordered partitions (not necessarily of the same integers). Furthermore, let 
\begin{equation}\label{eq:b} b=b(g,g',\mu)=2g'-2-d\cdot(2g-2)+\sum_{i}\ell(\mu^i)-|\mu^i|.\end{equation}
For non-negative integers $k,l,m$, such that $k+l+m=b$, we define a \textit{triply mixed factorisation of type} $(g,d,\mu,k,l,m)$ to be tuple $(\sigma_1,\dots,\sigma_n,\tau_1,\dots,\tau_{b},\alpha_1,\beta_1,\ldots,\alpha_g,\beta_g)$, such that
\begin{enumerate}
\item $\sigma_i,\tau_i,\alpha_i,\beta_i\in S_d$,
\item $\sigma_1\cdots\sigma_n\tau_1\cdots\tau_{b}=[\alpha_1,\beta_1]\cdots [\alpha_g,\beta_g]$,
\item $\mathcal{C}(\sigma_i)=\mu^i$ and the $\tau_j$ are transpositions,
\item[(6)] for  $\tau_i=(s_i\,t_i)$ with $s_i<t_i$, we have
\begin{itemize}
\item $t_i\le t_{i+1}$ for $i=k+1,\dots,k+l-1$,
\item $t_i<t_{i+1}$ for $i=k+l+1,\dots,k+l+m-1$.
\end{itemize}
\end{enumerate}
If in addition, we have
\begin{itemize}
\item[(4)] $\langle\sigma_1,\dots,\sigma_n,\tau_1,\dots,\tau_{b},\alpha,\beta\rangle$ acts transitively on $\{1,\dots,d\}$,
\end{itemize}
we call the factorisation \textit{connected}. We denote by $M^{\bullet}(g,d,\mu,k,l,m)$ the set of triply mixed factorisations of type $(g,d,\mu,k,l,m)$ and by $M^{\circ}(g,d,\mu,k,l,m)$ the set of connected triply mixed factorisations of type $(g,d,\mu,k,l,m)$. Then we define the \textit{triply mixed Hurwitz number} by
\begin{equation}
H^{g,d;\bullet}_{g';k,l,m}(\mu)=\frac{1}{d!}|M^{\bullet}(g,d,\mu,k,l,m)|
\end{equation}
and the \textit{connected triply mixed Hurwitz numbers} by
\begin{equation}
H^{g,d}_{g';k,l,m}(\mu)=\frac{1}{d!}|M^{\circ}(g,d,\mu,k,l,m)|.
\end{equation}
\end{definition}

\begin{remark}
Triply mixed Hurwitz numbers are interpolations between classical, monotone and strictly monotone Hurwitz numbers. Namely, taking $l=m=0$ yields classical Hurwitz numbers, $k=m=0$ yields monotone Hurwitz numbers and $l=m=0$ yields strictly monotone Hurwitz numbers.
\end{remark}

As mentioned before, we study several specifications of triply mixed Hurwitz numbers. For convience, we introduce additional notation distinguishing those cases:
\begin{enumerate}
\item In \cref{sec:quasimodularity}, we study triply mixed Hurwitz numbers for target surfaces of genus $1$. We abbreviate $H^{1,d}_{g';k,l,m}(\mu)$ by $H^{d}_{g';k,l,m}(\mu)$.
\item In \cref{sec:refined}, we study triply mixed Hurwitz numbers with target surfaces of genus $1$ for the following special cases:
\begin{itemize}
\item Let $k=m=0$, then we denote $H_{\le,g}^{d}=H_{g;0,b(g,\mu),0}^{d}(\mu)$ for $\mu=()$. This is the monotone case.
\item For $k=l=0$, we denote $H_{<,g}^{d}=H_{g;0,0,b(g,\mu)}^{d}(\mu)$ for $\mu=()$. This is the strictly monotone case.
\end{itemize}
\item In \cref{sec:quan}, we study two cases of triply mixed Hurwitz numbers for target surfaces of arbitrary genus and $\mu=((\mu_1,\dots,\mu_n))$, i.e. one arbitrary but fixed ramification profile:
\begin{itemize}
\item Let $k=m=0$, then we denote $\vec{H}_{\le,g'}^{g; \bullet}(\mu_1,\dots,\mu_n)=\vec{H}^{g',d;\bullet}_{g;0,b(g,\mu),0}(\mu)$. This is the monotone case.
\item For $k=l=0$, we denote $\vec{H}_{<,g'}^{g; \bullet}(\mu_1,\dots,\mu_n)=\vec{H}^{g',d;\bullet}_{g;0,0,b(g,\mu)}(\mu)$. This is the strictly monotone case.
\end{itemize}
\item In \cref{sec:refine}, we again study two cases of triply mixed Hurwitz numbers for target surfaces of genus $0$ and $\mu=(\mu^1,\mu^2)$, i.e. two arbitrary but fixed ramification profiles for two cases
\begin{itemize}
\item Let $k=m=0$, then we denote $h_{g}^{\le}(\mu^1,\mu^2)=H^{0,d}_{g;0,b(g,\mu),0}(\mu)$. This is the monotone case.
\item For $k=l=0$, we denote $h_{g}^{<}(\mu^1,\mu^2)=H^{0,d;\bullet}_{g;0,0,b(g,\mu)}(\mu)$. This is the strictly monotone case.
\end{itemize}
\end{enumerate}

\section{Quasimodularity of triply mixed Hurwitz numbers}\label{sec:quasimodularity}
Fix $g'\geq 2$. The generating series 
$$\sum_{d\geq 1}H^{d;\bullet}_{g';k,0,0}(\mu)q^d$$
of ordinary Hurwitz numbers with $g=1$ is known to be a quasimodular form (recall $q=e^{2\pi i \tau}$). This was observed by Dijkgraaf in the simplest case ($\mu=()$), rigorously proved by Kaneko and Zagier and follows in full generality from the Bloch--Okounov theorem as noted by Eskin and Okounkov \cite{Dij95, KZ95, EO01}. In the section we extend this result to the generating series of triply mixed Hurwitz numbers with $g=1$. We begin by expressing triply mixed Hurwitz numbers in terms of shifted symmetric functions.

 Let $h_n$ and $e_n$ be the complete homogeneous symmetric polynomial respectively elementary symmetric polynomial of degree $n$. 
\begin{proposition}\label{prop:hurwitzcentralcharacter}
Let $g,g'\ge0$, $\mu$ a tuple of partitions and $k+l+m=b$ with $b=b(g,g',\mu)$ given by \cref{eq:b}. Then, we have
\[\label{eq:hurwitzcentralcharacter}
H^{g,d;\bullet}_{g';k,l,m}(\mu) = \sum_{\lambda\vdash d} \left(\frac{\dim\lambda}{d!}\right)^{2-2g}f_\mu(\lambda) f_{(2)}(\lambda)^{k} h_l(\mathrm{cont}_\lambda) e_m(\mathrm{cont}_\lambda),
\]
where the sum is over all partitions of size $d$. 
\end{proposition}
\begin{proof}
First, we rewrite the triply mixed Hurwitz number in terms of the center of the group algebra $\mathcal{Z}_d$, see \cref{sec:centralcharacter} for the notation in this proof. 
Observe that
\[h_k(\Xi_d) = \sum_{\substack{2\leq t_1\leq\ldots\leq t_k\leq d\\s_i<t_i}}(s_1\,t_1)\cdots(s_k\,t_k) \quad \text{and} \quad  e_k(\Xi_d) = \sum_{\substack{2\leq t_1<\ldots< t_k\leq d\\s_i<t_i}}(s_1\,t_1)\cdots(s_k\,t_k).\]
Hence,
\[H^{d;\bullet}_{g';k,l,m}(\mu) = \frac{1}{d!}[C_e] \mathfrak{K}^gC_{\mu^1}\cdots C_{\mu^n} C_{(2)}^k h_l(\Xi_d) e_m(\Xi_d). \]
Observe that $\chi^\lambda(\sigma) = \chi^\lambda(e)\omega^\lambda(\sigma)$. Hence,  for $\sigma \in S_d$ the Schur orthogonality relation can be written in the unusual form $\sum_{\lambda} \chi^\lambda(e)^2\omega^\lambda(\sigma) = \delta_{e \sigma}|S_d|$. We find
\begin{align}
H^{d;\bullet}_{g';k,l,m}(\mu) &= \sum_{\lambda\vdash d}\left(\frac{\dim\lambda}{d!}\right)^2 \omega^\lambda\left(\mathfrak{K}^gC_{\mu^1}\cdots C_{\mu^n} C_{(2)}^k h_l(\Xi_d) e_m(\Xi_d)\right)\\
&= \sum_{\lambda\vdash d} \left(\frac{\dim\lambda}{d!}\right)^{2-2g}f_\mu(\lambda)f_{(2)}(\lambda)^k h_l(\mathrm{cont}_\lambda) e_m(\mathrm{cont}_\lambda). \qedhere
\end{align}
\end{proof}

\begin{definition} Let $t_{\mu}=\prod_{i,j} t_{i,\mu^i_j}$ be a formal variable. Define the \textit{Hurwitz potential} by
\[\mathfrak{H}^\bullet = \sum_{} H^{d;\bullet}_{g';k,l,m}(\mu) t_{\mu} \frac{u^k}{k!}v^lw^mq^d,\]
where the sum is over all $k,l,m,d,\mu$ for which $H^{d;\bullet}_{g';k,l,m}(\mu)$ is defined. Analogously define the connected Hurwitz potential $\mathfrak{H}$. 
\end{definition}
\begin{remark}\label{rk:potential} By a standard argument the Hurwitz potential and the connected Hurwitz potential are related by 
\[\exp \mathfrak{H} = 1+ \mathfrak{H}^\bullet.\]
\end{remark}

\begin{theorem}
\label{thm:genquasi}
Let $g'\geq 2$. Then
\begin{equation}\label{eq:qgs}\sum_{d=1}^\infty H^{d}_{g';k,l,m}(\mu) q^d\end{equation} is a quasimodular form of mixed weight $\leq 6g'-6+\sum_{i}4\ell(\mu^i)-2|\mu^i|$. Moreover, for fixed $b=k+l+m$ the top weight parts of 
\[2^{l+m+\delta_{l,0}+\delta_{m,0}-2}\sum_{d=1}^\infty H^{d;}_{g';k,l,m}(\mu) q^d\]
ranging over all $k,l,m\geq 0$ are equal. 
\end{theorem}
\begin{proof}
Observe that $f_\mu(\lambda)f_{(2)}(\lambda)^k h_l(\mathrm{cont}_\lambda) e_m(\mathrm{cont}_\lambda)$ is a shifted symmetric polynomial of mixed weight at most
\begin{equation}\sum_{i}\left(|\mu^i|+\ell(\mu_i)\right)+3(k+l+m) 
\end{equation}
by the results in \cref{sec:weights}. By \cref{eq:b} this weight equals 
\begin{equation}\label{eq:maxweight}6g-6+\sum_{i}4\ell(\mu^i)-2|\mu^i|.\end{equation}
Observe that evaluating $f_\mu$, $h_l(\mathrm{cont})$ or $e_m(\mathrm{cont})$ at the empty partition yields $0$ unless $\mu$ is empty respectively $l=0$ or $m=0$. In other words, using \cref{eq:hurwitzcentralcharacter} to define Hurwitz numbers for $d=0$, one obtains $H^{0}_{g';k,l,m}(\mu)=1$ if $\mu=()$ and $k=l=m=0$ and $0$ else. 
 Hence, \cref{prop:hurwitzcentralcharacter} and \cref{rk:potential} imply that
$$\mathfrak{H} = \log\left(q^{1/24}\eta(\tau)^{-1}\sum\langle f_\mu f_{(2)}^k h_l(\mathrm{cont}) e_m(\mathrm{cont}) \rangle_q  t_{\mu} \frac{u^k}{k!}v^lw^m\right).$$
By the Bloch--Okounov theorem $\langle f_\mu f_{(2)}^k h_l(\mathrm{cont}) e_m(\mathrm{cont}) \rangle_q$ is a quasimodular form of weight at most given by \cref{eq:maxweight}. Taking a formal Taylor expansion, quasimodularity of the generating series in \cref{eq:qgs} follows. The second part of the statement follows directly from \cref{prop:2d}. 
\end{proof}
\begin{remark}
In case $g'=1$ the series in \cref{eq:qgs} equals $-\log(q^{-1/24}\eta(\tau)).$ This is not a quasimodular form, but it is a primitive of a quasimodular form. Namely, its derivative equals up to a constant the Eisenstein series of weight $2$. 
\end{remark}

\section{Refined quasimodularity and tropical covers}\label{sec:refined}
In this section, we continue the study of the series
\begin{equation}
H_{\le,g}=\sum H_{\le,g}^{d}q^d\quad\mathrm{and}\quad H_{<,g}=\sum H_{<,g}^{d}q^d
\end{equation}
by expressing them as a finite sum of quasimodular forms, with each summand corresponding to a combinatorial type of the source curve of tropical covers.

\subsection{Tropical monotone and Grothendieck dessins d'enfants elliptic covers}
In this section, we express the numbers $H_{\le,g}^{d}(\mu)$ and $H_{<,g}^{d}(\mu)$ in terms of tropical covers of the tropical elliptic curve $E_{trop}$, which is a circle with a point $p_0$. As for the tropical projective line, we may add additional $2$ valent vertices to $E_{trop}$.

\begin{definition}
We fix $g\ge0$, an orientation on $E_{trop}$ and points $p_1,\dots,p_{2g-2}$ of $E_{trop}$, such that $p_0,p_1,\dots,p_{2g-2}$ is ordered according to the orientation. Let $\pi:\Gamma\to E_{trop}$ be a tropical cover of genus $g$ and degree $d$, such that $\pi^{-1}(p_0)$ does not contain any vertices and where $\Gamma$ has at most $2g-2$ vertices $v_1,\dots,v_n$, $n\le2g-2$. We require $v_i\in\pi^{-1}(p_i)$ for $i\in[n]$. We set $\lambda_i=\mathrm{val}(v_i)+2g(v_i)-2$ and obtain a composition $\lambda(\pi)=(\lambda_1,\dots,\lambda_n)$. If $|\lambda|=2g-2$, we call $\pi$ a \textit{monotone elliptic tropical cover} of type $(g,d)$ and denote by $\Gamma(E_{trop},g,d)$ the set of all monotone elliptic tropical covers of type $(g,d)$.\par
We further associate two multiplicities to each cover $\pi\in\Gamma(E_{trop},g,d)$, one corresponding to the monotone case and one to the strictly monotone case: 
\begin{align}
\mathrm{mult}_\le(\pi)&=\frac{1}{|\mathrm{Aut}(\pi)|}\frac{1}{\ell(\lambda(\pi))!}\prod_{v \in V(\Gamma)} m_v \prod_{e \in E(\Gamma)} \omega_e
\\
\mathrm{mult}_<(\pi)&=\frac{1}{|\mathrm{Aut}(\pi)|}\frac{1}{\ell(\lambda(\pi))!}\prod_{v \in V(\Gamma)} (-1)^{1 + \mathrm{val}(v)} m_v \prod_{e \in E(\Gamma)} \omega_e,
\end{align}
where for each vertex $v_i$, let $\textbf{x}^{+}$ (resp. $ \textbf{x}^-$) be the right-hand (resp. left-hand) side weights with respect to the orientation on $E_{trop}$. The multiplicity $m_{v_i}$ of $v_i$ is defined to be
\begin{align}
m_{v_i} = &(\lambda_i-1)!|\mathrm{Aut}(\textbf{x}^{+})||\mathrm{Aut}(\textbf{x}^{-})|\\
&\sum_{g_1^i+g_2^i=g(v_i)}\cor{ \!\! \tau_{2g^i_2 - 2}(\omega) \!\! }_{g^i_2}^{\mathbb{P}^1, \circ} \cor{ \!\!\textbf{x}^+,   \tau_{2g^i_1 - 2 + \ell(\textbf{x}^+) + \ell(\textbf{x}^-)}(\omega) , \textbf{x}^-\!\!  }^{\mathbb{P}^1,\circ}_{g^i_1}.
\end{align}
\end{definition}

\begin{remark}
We note that the multiplicity of the vertex is defined in a similar manner as in \cref{thm:trop}. This is due to our construction below, which glues tropical covers of $\mathbb{P}^1_{trop}$ to tropical covers of $E_{trop}$.
\end{remark}

This yields the following correspondence theorem, which we prove (below) in \cref{sec:proofrefined}.

\begin{theorem}
\label{thm:tropel}
Fix $g\ge0$ and $d>0$. Then, we have the following identities
\begin{align}
H_{\le,g}^{d}&=\sum_{\pi\in \Gamma(E_{trop};g,d)}\mathrm{mult}_{\le}(\pi),\\
H_{<,g}^{d}&=\sum_{\pi\in \Gamma(E_{trop};g,d)}\mathrm{mult}_{<}(\pi).
\end{align}
\end{theorem}

\subsubsection{Proof of \cref{thm:tropel}}
\label{sec:proofrefined}
In this section, we prove \cref{thm:tropel}. We only work out the details for the monotone case as the strictly monotone case is completely parallel. Before starting with the proof, we make the following remark.

\begin{remark}
In \cite{BBBMmirror}, a similar statement was proved in theorem 2.13. More precisely, a correspondence theorem expressing simple covers of elliptic curves (i.e. no (strict) monotonicity conditions and only simple ramification) in terms of tropical covers. We point out that while the idea of our proof of \cref{thm:tropel} is similar to the idea of the proof of theorem 2.13 in \cite{BBBMmirror}, the proof itself is still very different in nature. This is due to an important technical sublety: In \cite{BBBMmirror}, it was possible to construct a tropical elliptic cover directly from a factorisation in the symmetric group and vice versa. In our setting, this is no longer possible due to the nature of \cref{thm:trop} as it is derived using the bosonic Fock space. However, we will show that the relation between factorisations and tropical covers is still close enough to derive our correspondence theorem in \cref{thm:tropel}.
\end{remark}

First, we recall the definition of $H^d_{\le,g}$. We count tuples of certain permutations $(\tau_1,\dots,\tau_{2g-2},\alpha,\beta)$, such that
\begin{equation}
\tau_{2g-2}\cdots\tau_1=\alpha\beta\alpha^{-1}\beta^{-1}.
\end{equation}
However, this is equivalent to
\begin{equation}
\tau_{2g-2}\cdots\tau_1\beta=\alpha\beta\alpha^{-1}.
\end{equation}

Thus, letting $\sigma_1=\beta$ and $\sigma_2=\alpha\beta\alpha^{-1}$, we see that $H^d_{\le,g}$ is equal to $\frac{1}{d!}$ times the number of tuples $(\sigma_1,\tau_1,\dots,\tau_{2g-2},\sigma_2,\alpha)$, such that
\begin{enumerate}
\item $\sigma_1,\sigma_2,\alpha,\tau_i\in S_d$,
\item $\mathcal{C}(\tau_i)=(2,1,\dots,1)$, $\mathcal{C}(\sigma_1)=\mathcal{C}(\sigma_2)$,
\item $\tau_{2g-2}\cdots\tau_1\sigma_1=\sigma_2$,
\item $\alpha\sigma_1\alpha^{-1}=\sigma_2$ (we note that thus $\mathcal{C}(\sigma_1)=\mathcal{C}(\sigma_2)$,
\item $\tau_i$ satisfy the monotonicity condition,
\item the group $\langle\sigma_1,\tau_1,\dots,\tau_{2g-2},\sigma_2,\alpha\rangle$ acts transitively on $\{1,\dots,d\}$.
\end{enumerate}
Observe that we count tuples very similar to the monotone double Hurwitz numbers framework for $\mu=\nu$. We note that in the above description the group $\langle\sigma_1,\tau_1,\dots,\tau_{2g-2},\sigma_2\rangle$, i.e. without the generator $\alpha$, might have several orbits acting on $\{1,\ldots,d\}$. By capturing this data, we can make the full transition to monotone double Hurwitz numbers. We now make this more precise.\par 
For a monotone base $0$ factorisation $(\sigma_1,\tau_1,\dots,\tau_b,\sigma_2)$, we consider the orbits of the action of the group $\langle\sigma_1,\tau_1,\dots,\tau_b,\sigma_2\rangle$ on $\{1,\dots,d\}$. The orbits then naturally yield connected monotone base $0$ factorisations $(\sigma_1^{(i)},\tau_{i(1)},\dots,\tau_{i(l_k)},\sigma_2^{(i)})$ of type $(g_i,\eta_1^i,\eta_2^i)$ for $i=1,\dots,n$ (for some arbitrary $n$), where $\sigma_j^{(i)}$ have pairwise disjoint orbits with $\prod\sigma_j^{(i)}=\sigma_j$ for $j=1,2$, such that
\begin{enumerate}[(i)]
\item the $\bigcup\eta_1^i=\bigcup\eta_2^i=\mathcal{C}(\sigma_1)=\mathcal{C}(\sigma_2)$,
\item the sets $\{i(1),\dots,i(l_k)\}$ are pairwise disjoint and $\bigcup\{i(1),\dots,i(l_k)\}=[b]$,
\item we have $\sum g_i=g+n-1$.
\end{enumerate}
We call these condition (i)---(iii) the \textit{orbit conditions}. Now fix an unordered tuple $\Lambda=((g_i,\eta_1^i,\eta_2^i))$ satisfying these conditions (i) and (iii). We then define $h_g(\mu,\Lambda)$ to be $\frac{1}{d!}$ times the number of monotone base $0$ factorisations whose orbits yield the data $\Lambda$.\par 
Our strategy now is as follows: To each monotone base $0$ factorisation $(\sigma_1,\tau_1,\dots,\tau_b,\sigma_2)$, we want to associate $\alpha\in S_d$, such that $\alpha\sigma_1\alpha^{-1}=\sigma_2$ and enumerate those $\alpha$. This number only depends on the data $\Lambda$ and is actually encoded in the tropical pictures.

\paragraph{Cutting elliptic covers}
We fix a monotone elliptic tropical cover $\pi:\Gamma\to E_{trop}$ of type $(g,d)$. We consider $\pi^{-1}(p_0)$ and collect the weights in the preimage of $p_0$ in the partition $\Delta$.\par 
We cut the elliptic curve $E_{trop}$ at $p_0$ open and cut the source curve in the pre-image. This way, we obtain a tropical cover $\pi':\Gamma'\to\mathbb{P}^1_{trop}$ in $\Gamma(\mathbb{P}^1_{trop},g;\Delta,\Delta,\lambda(\pi))$ for some non-negative integer $g$, which we call the \textit{cut-cover associated to $\pi$}. This cover may be disconnected. The connected components yield tropical covers in $\Gamma(\mathbb{P}^1_{trop},g_i;\Delta_i,\Delta^j,\lambda^i)$, such that
\begin{enumerate}[(i)]
\item we have $\bigcup \Delta_i=\bigcup\Delta^j=\Delta$,
\item we have $\bigcup\lambda^i=\lambda$,
\item the genera satisfy $\sum g_i=g+n-1$.
\end{enumerate}

We call these conditions the \textit{component conditions}. Observe that the component conditions (i) and (iii) coincide with orbit conditions (i) and (iii) above for monotone factorisations.\par 
We now fix an unordered tuple $\Lambda=((g_i,\Delta_i,\Delta^i))$ satisfying conditions (i) and (iii) and denote by $\Gamma(\mathbb{P}^1_{trop},g;\Delta,\Delta,\lambda;\Lambda)$ the set of all tropical covers in $\Gamma(\mathbb{P}^1,g;\Delta,\Delta,\lambda)$, such that their connected components yield the data $\Lambda$.

By the inclusion-exclusion principle, we obtain the following lemma, which states that the data of connected components of monotone factorisations are captured by connected components in the tropical covers of \cref{thm:trop}:

\begin{lemma}
\label{lem:tropor}
Let $g$ be a non-negative integer and $\Delta$ a partition of some positive integer and $b=2g-2+2\ell(\mu)$. Further fix an unordered tuple $\Lambda=((g_i,\eta_1^i,\eta_2^i))$ satisfying the component conditions \upshape{(i)} and \upshape{(iii)}. Then we obtain
\begin{equation}
h_g(\Delta,\Lambda)=\sum_{\lambda\vdash b}\sum_{\pi\in\Gamma(\mathbb{P}^1_{trop},g;\Delta,\Delta,\lambda;\Lambda)}\mathrm{mult}(\pi).
\end{equation}
\end{lemma}

\paragraph{Regluing ellipic covers}
The role of monotone elliptic covers is essentially to encode ways to find permutations $\alpha$ for monotone base $0$ factorisations, which enrich them to be a factorisation contributing to $H_{\le,g}^d$. We now make this more precise.\par 
Let $(\sigma_1,\tau_1,\dots,\tau_b,\sigma_2)$ be a monotone base $0$ factorisation of type $(g,\mu,\mu)$. We fix a permutation $\alpha$, such that $\alpha\sigma_1\alpha^{-1}=\sigma_2$. We observe that conjugation by $\alpha$ maps the cycles of $\sigma_1$ bijectively to the cycles of $\sigma_2$. In other words, when we choose a labelling of the cycles of $\sigma_1$ and $\sigma_2$ by $1,\dots,\ell(\mu)$, then $\alpha$ induces a bijection $I_{\alpha}:\{1,\dots,\ell(\mu)\}\to\{1,\dots,\ell(\mu)\}$, such that for $j\in\{1,\dots,\ell(\mu)\}$ the length of the cycle of $\sigma_1$ labeled $j$ coincides with the length of cycle of $\sigma_2$ labeled $I_{\alpha}(j)$. This corresponds to a gluing process in the tropical setting.\par 
Let $\pi':\Gamma'\to\mathbb{P}^1_{trop}$ in $\Gamma(\mathbb{P}^1,g;\mu,\mu,\lambda)$ for some $\lambda\vdash b$. We label the left ends of $\Gamma'$ by $1,\dots,\ell(\mu)$ and the right ends of $\Gamma'$ as well. We fix a bijection $I:\{1,\dots,\ell(\mu)\}\to\{1,\dots,\ell(\mu)\}$, such that the weight of the left weight labeled $j$ is the same as the weight of right end labeled $I(j)$. We now glue the ends of $\mathbb{P}^1_{trop}$ together to obtain $E_{trop}$, where $p_0$ is gluing point. We also glue the source curve $\Gamma'$ according to $I$, i.e  we glue the left end labeled $j$ to the right end labeled $I(j)$. This way, we obtain a monotone elliptic tropical cover of type $(g',d)$ for $g'=\frac{b+2}{2}$.

\begin{definition}
\label{def:nnumber}
Let $\pi:\Gamma\to E_{trop}$ a monotone elliptic tropical cover of type $(g,d)$ and $\pi':\Gamma'\to\mathbb{P}^1_{trop}$ be the associated cut-cover, where $\pi'\in\Gamma(\mathbb{P}^1,g';\mu,\mu,\lambda;\Lambda)$ for some $g'$ and $\lambda\vdash b$. We fix a monotone base $0$ factorisation $(\sigma_1,\tau_1,\dots,\tau_b,\sigma_2)$ of type $(g,\mu,\mu)$. We label the left ends of $\Gamma'$ by $\sigma_1$ and the right ends by $\sigma_2$. We denote by $n_{\pi,\pi'}$ the number of $\alpha\in S_d$, where $\alpha\sigma_1\alpha^{-1}=\sigma_2$ and such that the associated gluing of $\pi'$ induced by $\alpha$ yields $\pi$.
\end{definition}

This situation was analysed in \cite{BBBMmirror}.

\begin{proposition}[\cite{BBBMmirror}]
\label{prop:tropenum}
For an elliptic cover $\pi$ with tuple $\Delta=(m_1,\dots,m_r)$ over the base point $p_0$ and its cut-cover $\pi'$, we have
\begin{equation}
n_{\pi,\pi'}=\frac{|\mathrm{Aut}(\pi')|}{|\mathrm{Aut}(\pi)|}m_1\cdots m_r.
\end{equation}
\end{proposition}

Moreover, we have the following lemma. Recall that a monotone elliptic tropical cover is always a connected cover by definition.

\begin{lemma}
\label{lem:tropcon}
Let $\pi$ be a monotone elliptic tropical cover and $\pi'$ be the associated cut-cover. Furthermore, let $(\sigma_1,\tau_1,\dots,\tau_b,\sigma_2)$ be a monotone factorisation, whose orbits yield the same data $\Lambda$ as the cut-cover $\pi'$. Let $\alpha$ be a permutation as in \cref{def:nnumber}, then
\begin{equation}
\left\langle\sigma_1,\tau_1,\dots,\tau_b,\sigma_2,\alpha\right\rangle
\end{equation}
is a transitive subgroup of $S_d$.
\end{lemma}

\begin{proof}
This follows immediately from the fact that $\pi$ is a connected cover and $\alpha$ joins the connected components.
\end{proof}

We are now ready to prove \cref{thm:tropel}.

\begin{proof}[Proof of \cref{thm:tropel}]
In the beginning of this section, we have seen that for fixed $g,d$ and $b=2g-2$, we have $H_{\le,g}^d$ is equal to $\frac{1}{d!}$ times $(\sigma_1,\tau_1,\dots,\tau_b,\sigma_2,\alpha)$, such that $(\sigma_1,\tau_1,\dots,\tau_b,\sigma_2)$ is a montone factorisation, $\alpha\sigma_1\alpha^{-1}=\sigma_2$ and $\langle\sigma_1,\tau_1,\dots,\tau_b,\sigma_2,\alpha\rangle$ is a transitive subgroup of $S_d$. We associated to each monotone factorisation $\kappa=(\sigma_1,\tau_1,\dots,\tau_b,\sigma_2)$ the number $n(\kappa)$ of permutations $\alpha$, such that $(\sigma_1,\tau_1,\dots,\tau_b,\sigma_2,\alpha)$ contributes to $H_{\le,g}^d$. Thus, we obtain
\begin{equation}
H_{\le,g}^d=\frac{1}{d!}\sum n(\kappa),
\end{equation}
where we sum over all monotone factorisations with $b$ transpositions in $S_d$. We have also seen that several monotone factorisations yield the same number $n(\kappa)$, i.e. those satisfying the same orbit conditions $\Lambda$. We thus denote by $n(\Lambda)$ the number $n(\kappa)$ for all monotone factorisations whose orbits yield the data $\Lambda$. We now group those together, i.e. we obtain
\begin{equation}
H_{\le,g}^d=\sum h_{g'}(\mu,\Lambda)\cdot n(\Lambda),
\end{equation}
where we sum over all non-negative integer $g'\le g$, partitions $\mu$ of $d$ and tuples $\Lambda$ satisfying the orbit conditions (i) and (iii). We now analyse each summand $h_{g'}(\mu,\Lambda)\cdot n(\Lambda)$. In fact, we want to prove that
\begin{equation}
h_{g'}(\mu,\Lambda)\cdot n(\Lambda)=\sum_{\lambda\vdash b}\sum_{\pi'}\sum_{\pi} \mathrm{mult}(\pi')\cdot n(\pi,\pi'),
\end{equation}
where the second summand is over all cover $\pi'\in\Gamma(\mathbb{P}^1_{trop},g';\mu,\mu,\lambda;\Lambda)$ and the second summand over all monotone elliptic tropical covers $\pi$, such that $\pi'$ is their cut-cover.\par
By \cref{lem:tropor}, we have
\begin{equation}
h_g(\mu,\Lambda)\cdot n(\Lambda)=\sum_{\lambda\vdash b}\sum_{\pi'\in\Gamma(\mathbb{P}^1_{trop},g;\mu,\mu,\lambda;\Lambda)}\mathrm{mult}(\pi')\cdot n(\Lambda)=\sum_{\lambda\vdash b}\sum_{\pi'\in\Gamma(\mathbb{P}^1_{trop},g;\mu,\mu,\lambda;\Lambda)}(\mathrm{mult}(\pi')\cdot n(\Lambda)).
\end{equation}
We now observe that each $\alpha$, which contributes to $n(\Lambda)$ for a given monotone factorisation, whose orbits yield $\Lambda$ contributes to a gluing of $\pi'$ to an monotone elliptic tropical cover. Thus, we have by \cref{lem:tropcon}
\begin{equation}
n(\Lambda)=\sum_{\pi}n_{\pi,\pi'},
\end{equation}
where we sum over all monotone elliptic tropical covers $\pi$ whose cut-cover is $\pi'$. Thus, by \cref{prop:tropenum} we have
\begin{align}
h_g(\mu,\Lambda)\cdot n(\Lambda)=&\sum_{\lambda\vdash b}\sum_{\pi'\in\Gamma(\mathbb{P}^1_{trop},g;\mu,\mu,\lambda;\Lambda)}(\mathrm{mult}(\pi')\cdot \sum_{\pi}n_{\pi,\pi'})\\
=&\sum_{\lambda\vdash b}\sum_{\pi'\in\Gamma(\mathbb{P}^1_{trop},g;\mu,\mu,\lambda;\Lambda)}\sum_{\pi}\mathrm{mult}(\pi')\cdot n_{\pi,\pi'},
\end{align}
as desired. We further see that
\begin{align}
\mathrm{mult}(\pi')\cdot n_{\pi,\pi'}=\frac{1}{|\mathrm{Aut}(\pi')|}\frac{1}{\ell(\lambda)!}\prod_{v \in V(\Gamma)} m_v \prod_{e \in E(\Gamma)} \omega_e\cdot\prod_{i=1}^{\ell(\mu)}\mu_i\cdot\frac{|\mathrm{Aut}(\pi')|}{|\mathrm{Aut}(\pi)|}=\mathrm{mult}(\pi)
\end{align}
and obtain
\begin{align}
H_{\le,g}^d=\sum_{\Lambda} h_{g'}(\mu;\Lambda)\cdot n(\Lambda)&=\sum_{\Lambda}\sum_{\lambda\vdash b}\sum_{\pi'}\sum_{\pi} \mathrm{mult}(\pi')\cdot n(\pi,\pi')\\
&=\sum_{\Lambda}\sum_{\lambda\vdash b}\sum_{\pi'}\sum_{\pi}\mathrm{mult}(\pi).
\end{align}
As $\pi'$ and $\Lambda$ are determined by $\pi$, we can ommit those summands and just sum over all monotone elliptic tropical covers $\pi$. This yields
\begin{equation}
H_{\le,g}^d=\sum_{\lambda\vdash b}\sum_{\pi}\mathrm{mult}(\pi),
\end{equation}
where we sum over all monotone elliptic tropical covers $\pi$ of type $(g,d)$ as desired.
\end{proof}

%
%

\subsection{Refined quasimodularity}
\begin{definition}
We fix a combinatorial type $G$ of a tropical curve $\Gamma$ with $n:=|V(\Gamma)|\le 2g(\Gamma)-2$. Moreover, we fix an orientation on $E_{trop}$ and a linear ordering $\Omega$ on the vertices of $G$. We denote by $v_i$ the $i-$th vertex according to $\Omega$. We further choose points $p_1,\dots,p_{2g-2}$ on $E_{trop}$, such that they are linearly ordered along the orientation we chose on $E_{trop}$. We further fix a series of integers $\underline{g'}=(g_1,\dots,g_n)$.\par 
We denote by $\Gamma(G,\Omega;g,\underline{g'})$ the set of all covers $\pi\in\Gamma(E_{trop},g,d)$ for some $d\in\mathbb{N}$, such that
\begin{enumerate}
\item for $\pi:\Gamma\to E_{trop}$ the combinatorial type of $\Gamma$ is $G$,
\item we have $\pi(v_i)=p_i$,
\item we have $g(v_i)=g_i$.
\end{enumerate}
Moreover, we associate two generating series to each combinatorial type $G$
\begin{align}
I_{\le,\underline{g'}}^{G,\Omega}\coloneqq\sum_{\pi\in\Gamma(G,\Omega;g,\underline{g'})}\mathrm{mult}_{\le}(\pi)q^{\mathrm{deg}(\pi)}\\
I_{<,\underline{g'}}^{G,\Omega}\coloneqq\sum_{\pi\in\Gamma(G,\Omega;g,\underline{g'})}\mathrm{mult}_{<}(\pi)q^{\mathrm{deg}(\pi)}
\end{align}
\end{definition}

\begin{remark}
We observe that by the correspondence theorem
\begin{equation}
\label{equ:identity}
H_{\le,g}=\sum_{(G,\Omega,\underline{g'})}I_{\le,\underline{g'}}^{G,\Omega}\quad\mathrm{and}\quad H_{\le,g}=\sum_{(G,\Omega,\underline{g})}I_{<,\underline{g'}}^{G,\Omega},
\end{equation}
where we sum over all combinatorial types $G$ on at most $n\le 2g-2$ vertices, orders $\Omega$ on $G$ and tuples $\underline{g'}=(g_1,\dots,g_n)$.
\end{remark}

\begin{theorem}
\label{thm:quasi}
For $g'\ge2$, the series $I_{\le,\underline{g'}}^{G,\Omega}$ and $I_{<,\underline{g'}}^{G,\Omega}$ are quasimodular forms of mixed weight less or equal $2\left(\sum_{i=1}^n g_i+|E(G)|\right)$.
\end{theorem}

\begin{proof}
This follows from \cite[Theorem 6.1,Corollary 8.4]{goujard2016counting}, where it is proved that the generating series associated to a tropical cover with target curve of fixed combinatorial type, order and fixed ramification profile is a quasimodular form whenever the multiplicity of the cover is a polynomial in the edge weights. The only thing to check in our case is that the local vertex multiplicities are polynomial, which is true as proved in \cite[Theorem 4.1]{goujard2016counting}.
\end{proof}

Combining \cref{equ:identity} and \cref{thm:quasi}, we obtain the following corollary (which also follows from \cref{thm:genquasi}):

\begin{corollary}
The generating series
\begin{equation}
\sum_{d\ge1} H^{d}_{\le,g}q^d\quad\mathrm{and}\quad\sum_{d\ge1} H^{d}_{<,g}q^d
\end{equation}
are quasimodular forms of mixed weight $\le 6g-6$.
\end{corollary}

\begin{proof}
The only thing left to prove is the weight part of the corollary. In order to see this, we consider $2\left(\sum_{i=1}^n g_i+|E(G)|\right)$ and observe $2|E(G)|=\sum_{i=1}^n\mathrm{val}(v_i),$ where we sum over all vertices of $G$. Moreover, by definition, we have $\mathrm{val}(v_i)=\lambda_i-2g_i+2$. We obtain
\begin{align}
2\left(\sum_{i=1}^n g_i+|E(G)|\right)&=
2\sum_{i=1}^ng_i+\sum_{i=1}^n(\lambda_i-2g_i+2)\\
&=\sum_{i=1}^n\lambda_i+2n=2g-2+2n\le 2g-2+2(2g-2)=6g-6,
\end{align}
where we used $\sum_{i=1}^n\lambda_i=2g-2$ and $n\le 2g-2$. This yields the weight as desired.
\end{proof}

\section{Quantum curve for (strictly) monotone base \texorpdfstring{$g$}{g} Hurwitz numbers} \label{sec:quan}
Motivated by the successful study of base $g$ Hurwitz numbers in \cite{LMSsimplequantum} and monotone Hurwitz numbers \cite{GGPNmonotone} with a view towards topological recursion, we connect these ideas by enumerating base $g$ Hurwitz numbers with monotonicity conditions. In particular, this section is devoted to deriving a quantum curve for this new enumerative problem.

Recall the connected labeled monotone base $g$ Hurwitz numbers $\vec{H}_{\leq,g'}^{g}(\mu_1,\dots,\mu_n)$ from \cref{sec:Hurwitz}.  

\begin{definition}
We define
\begin{align}
F_{\leq,g'}^{g}(x_1,\dots,x_n)=\sum_{\mu\in\mathbb{Z}_+^n}\vec{H}_{\leq,g'}^{g}(\mu_1,\dots,\mu_n)x_1^{\mu_1}\dots x_n^{\mu_n},
\end{align}
and analogously we define the generating series for the strictly monotone case $F_{<,h}^{g}(x_1,\dots,x_n)$.
\end{definition}

\begin{definition}
We define the partition function of the base $g$ monotone Hurwitz numbers, as the formal series in variables $x,\hbar$, given by
\begin{align}
Z_\leq^g=Z_\leq^g(x,\hbar)&=\exp\bigg[\sum_{h=0}^\infty\sum_{n=1}^{\infty}\frac{\hbar^{2g'-2+n}}{n!}F_{g'}^{\leq,g}(x,x,\dots,x)\bigg]\\
&=\exp\bigg[\sum_{h=0}^\infty\sum_{n=1}^{\infty}\frac{\hbar^{2g'-2+n}}{n!}\sum_{\mu\in\mathbb{Z}_+^n}\vec{H}_{g'}^{g}(\mu_1,\dots,\mu_n)x^{|\mu|}\bigg],
\end{align}
and analogously the partition function $Z_<^g=Z_<^g(x,\hbar)$ of the base $g$ strictly monotone Hurwitz numbers. 
\end{definition}

\begin{proposition}\label{PartitionFunktionStirling}
For the partition functions for the monotone and strictly monotone case, we have
\begin{align}
Z_\leq^g&=1+\sum_{d=1}^\infty\sum_{b=0}^\infty\stirling{d+b-1}{d-1}(d!)^{1-\chi}x^d\hbar^{b+d(1-\chi)}\\
&=1+\sum_{d=1}^\infty (d!)^{1-\chi}x^d\hbar^{d(1-\chi)}\prod_{j=1}^{d-1}\frac{1}{1-j\hbar}
\end{align}
and
\begin{align}
Z_<^g&=1+\sum_{d=1}^\infty\sum_{b=0}^{d-1}\fstirling{d}{d-b}(d!)^{1-\chi}x^d\hbar^{b+d(1-\chi)}\\
&=1+\sum_{d=1}^{\infty}(d!)^{1-\chi}x^d\hbar^{1-\chi}\prod_{j=1}^{d-1} (1-j\hbar),
\end{align}
where $\chi=2-2g$ and the equalities are understood in the sense of formal power series.
\end{proposition}
\begin{proof}
Since
\begin{align}
\sum_{g'=0}^\infty\sum_{n=1}^{\infty}\frac{\hbar^{2g'-2+n}}{n!}F_{\leq,g'}^g(x,x,\dots,x)=\sum_{g'=0}^\infty\sum_{n=1}^{\infty}\frac{\hbar^{2g'-2+n}}{n!}\sum_{\mu\in\mathbb{Z}_+^n}\vec{H}_{\leq,g'}^{g}(\mu_1,\dots,\mu_n)x^{|\mu|}
\end{align}
counts transitive (connected) monotone base $g$ factorisations, we can use the exponential formula and find the generating series for the not necessarily transitive factorisations
\begin{align}
Z_\leq^g=1+\sum_{g'=0}^\infty\sum_{n=1}^{\infty}\frac{\hbar^{2g'-2+n}}{n!}\sum_{\mu_1,\dots,\mu_n=1}^\infty\vec{H}_{\leq,g'}^{\bullet,g}(\mu_1,\dots,\mu_n)x^{|\mu|}.
\end{align}
Collecting all factorizations for given $d=|\mu|$, we get
\begin{align}
Z_\leq^g=1+\sum_{g'=0}^\infty\sum_{n=1}^{\infty}\frac{\hbar^{2g'-2+n}}{n!}\sum_{d=0}^\infty\bigg(\sum_{|\mu| = d}\vec{H}_{\leq,g'}^{\bullet,g}(\mu)\bigg)x^{d}.
\end{align}
Recall that $b=b(h,n,|\mu|)=2g'-2+n-|\mu|(2g-1)$, so we write
\begin{align}
\sum_{g'=0}^\infty\sum_{n=1}^{\infty}\frac{\hbar^{2g'-2+n}}{n!}\sum_{d=0}^\infty\bigg(\sum_{|\mu|=d}\vec{H}_{\leq,g'}^{\bullet,g}(\mu)\bigg)x^{d}&=\sum_{g'=0}^\infty\sum_{n=1}^{\infty}\sum_{d=0}^\infty\frac{\sum_{|\mu|=d}\vec{H}_{\leq,g'}^{\bullet,g}(\mu)}{n!}\hbar^{b}(x\hbar^{2g-1})^d\\
&=\sum_{g'=0}^\infty\sum_{n=1}^{\infty}\sum_{d=0}^\infty\frac{\sum_{|\mu|=d}\vec{H}_{\leq,g'}^{\bullet,g}(\mu)}{n!}\hbar^{b}(x\hbar^{1-\chi})^d.
\end{align}
Since $\vec{H}_{\leq,g'}^g(\mu)$ is non-zero if $b\geq 0$, we can rearrange the series by collecting all possible $g',n$ for a given $b$ and $d$. Viewing the partition function $Z^g_{\leq}$ as a series in $\mathbb{Q}[\![\hbar,x\hbar^{1-\chi}]\!]$, we find that the coefficient of $x^d\hbar^{b+d(1-\chi)}$ is precisely the number of monotone base $g$ factorizations of length $b$ in $S_d$, i.e.
\begin{align}
[x^d\hbar^{b+d(1-\chi)}]Z_\leq^g&=\frac{1}{d!}\#\left\{(\tau_1,\dots,\tau_b,\sigma,\alpha_1,\beta_1,\dots,\alpha_{g},\beta_{g})\left| \begin{aligned} &\tau_i \text{ monotone transpositions},\\ &\sigma, \alpha_1,\beta_1\dots\alpha_{g},\beta_{g}\in S_d, \\ &\sigma \tau_1\dots\tau_b=[\alpha_1,\beta_1]\dots[\alpha_{g},\beta_{g}]\end{aligned}\right. \right\}.
\end{align}
Since the $\alpha_i,\beta_i$ run over all elements in $S_d$ and using Lemma 17 of \cite{DDMmonotone}, we obtain
\begin{align}
[x^d\hbar^{b+d(1-\chi)}]Z^g&=(d!)^{2g-1}\#\{(\tau_1,\dots,\tau_b)| \tau_i\textrm{ monotone transpositions }\}\\
&=(d!)^{1-\chi}\stirling{d+b-1}{d-1}.
\end{align}
For the last equalitiy we find
\begin{align}
Z_\leq^g&=1+\sum_{d=1}^\infty\sum_{b=0}^\infty\stirling{d+b-1}{d-1}(d!)^{1-\chi}x^d\hbar^{b+d(1-\chi)}\\
&=1+\sum_{d=1}^\infty(d!)^{1-\chi}x^d\hbar^{d(1-\chi)}\sum_{b=0}^\infty\stirling{d+b-1}{d-1}\hbar^{b}\\
&=1+\sum_{d=1}^\infty(d!)^{1-\chi}x^d\hbar^{d(1-\chi)}\prod_{j=1}^{d-1}\frac{1}{1-j\hbar},
\end{align}
where we used the well-known identity in \cref{eq:stirling}.
\begin{align}
\sum_{N=0}^\infty \stirling{N}{K}\hbar^{N-K}=\prod_{j=1}^{K}\frac{1}{1-j\hbar}.
\end{align}
In the case of the strictly monotone Hurwitz numbers we do a similar calculation and view $Z_<^g$ as an element of $\mathbb{Q}[\![\hbar,x\hbar^{1-\chi}]\!]$. We find
\begin{align}
[x^d\hbar^{b+d(1-\chi)}]Z_<^g&=(d!)^{1-\chi}\#\{(\tau_1,\dots,\tau_b)| \tau_i\textrm{ strictly monotone transpositions }\}
\end{align}
and in particular $[x^d\hbar^{b+d(1-\chi)}]Z_<^g = 0$ for $b\geq d$. These tuples can be expressed by evaluating the elementary symmetric polynomials in the so-called Jucys–Murphy  elements. This also yields an enumeration by evaluating the same polynomials in the number of summands of the $i-$th Jucys-Murphy element, which yields (see e.g. \cite[Equation (4)]{KLSmonotonewedge})
\begin{align}
\#\{(\tau_1,\dots,\tau_b)| \tau_i\textrm{ strictly monotone transpositions }\}=\fstirling{d}{d-b}.
\end{align}
for $b \leq d$. Hence we have the assertion
\begin{align}
Z_<^g&=1+\sum_{d=1}^\infty\sum_{b=0}^{d-1}\fstirling{d}{d-b}(d!)^{1-\chi}x^d\hbar^{b+d(1-\chi)}\\
&=1+\sum_{d=1}^\infty(d!)^{1-\chi}x^d\hbar^{1-\chi}\sum_{b=0}^{d-1}\fstirling{d}{d-b}\hbar^d\\
&=1+\sum_{d=1}^{\infty}(d!)^{1-\chi}x^d\hbar^{1-\chi}\prod_{j=1}^{d-1} (1-j\hbar),
\end{align}
where the last equality follows by \cref{eq:stirling}. 
\end{proof}
\begin{theorem}
\label{thm:mono}
The partition function $Z_\leq^g$ satisfies the differential equation
\begin{align}
[\widehat{x}\widehat{y}^2+\widehat{y}+(\widehat{y}\widehat{x})^{2g}]Z_{\leq}^g=0,
\end{align}
where $\widehat{x}=x$ and $\widehat{y}=-\hbar\frac{\partial}{\partial x}$.
\end{theorem}
\begin{proof}
The recursion formula for the Stirling numbers of the second kind yields
\begin{align}
\stirling{d+b-1}{d-1}=(d-1)\stirling{d+b-2}{d-1}+\stirling{d+b-2}{d-2},
\end{align}
We multiply this equation by $\frac{(d!)^{2g}}{(d-1)!}x^d\hbar^{b+d(1-\chi)}$ and sum over $d\geq 1,b\geq 0$. For reasons of clarity, we first do the computations term by term before conflating them. For the term on the left hand side we have
\begin{align}
\sum_{\substack{d=1\\b=0}}^{\infty}\stirling{d+b-1}{d-1}&\frac{(d!)^{2g}}{(d-1)!}x^d\hbar^{b+d(1-\chi)}\\
&=\sum_{\substack{d=1\\b=0}}^\infty\stirling{d+b-1}{d-1} d(d!)^{1-\chi}x^d\hbar^{b+d(1-\chi)}\\
&=x\frac{\partial}{\partial x}\sum_{\substack{d=1\\b=0}}^\infty\stirling{d+b-1}{d-1} (d!)^{1-\chi}x^d\hbar^{b+d(1-\chi)}\\
&=x\frac{\partial}{\partial x}\bigg(1+\sum_{\substack{d=1\\b=0}}^\infty\stirling{d+b-1}{d-1} (d!)^{1-\chi}x^d\hbar^{b+d(1-\chi)}\bigg)\\
&=x\frac{\partial}{\partial x}Z_\leq^g,
\end{align}
where we used the fact that the derivative of the constant $1$ vanishes.
For the first expression on the right hand side we get
\begin{align}
\sum_{\substack{d=1\\b=0}}^{\infty}(d-1)\stirling{d+b-2}{d-1}&\frac{(d!)^{2g}}{(d-1)!}x^d\hbar^{b+d(1-\chi)}\\ 
&=\sum_{\substack{d=1\\b=0}}^{\infty}d(d-1)\stirling{d+b-2}{d-1}(d!)^{1-\chi}x^d\hbar^{b+d(1-\chi)}\\
&=x^2\hbar\frac{\partial^2}{\partial x^2}\sum_{\substack{d=1\\b=0}}^{\infty}\stirling{d+b-2}{d-1}(d!)^{1-\chi}x^d\hbar^{b-1+d(1-\chi)}.
\end{align}
Now, note that for $b=0$ we have $\stirling{d-2}{d-1}=0$. Hence
\begin{align}
x^2\hbar\frac{\partial^2}{\partial x^2}\sum_{\substack{d=1\\b=0}}^{\infty}\stirling{d+b-2}{d-1}&(d!)^{1-\chi}x^d\hbar^{b-1+d(1-\chi)}\\
&=x^2\hbar\frac{\partial^2}{\partial x^2}\sum_{\substack{d=1\\b=1}}^{\infty}\stirling{d+b-2}{d-1}(d!)^{1-\chi}x^d\hbar^{b-1+d(1-\chi)}\\
&=x^2\hbar\frac{\partial^2}{\partial x^2}\sum_{\substack{d=1\\b=0}}^{\infty}\stirling{d+b-1}{d-1}(d!)^{1-\chi}x^d\hbar^{b+d(1-\chi)}\\
&=x^2\hbar\frac{\partial^2}{\partial x^2}Z_\leq^g,
\end{align}
where we performed the shift $b^\prime=b-1$ in the second equality. For the last term we obtain
\begin{align}
\sum_{\substack{d=1\\b=0}}^{\infty}&\stirling{d+b-2}{d-2}\frac{(d!)^{2g}}{(d-1)!}x^d\hbar^{b+d(1-\chi)}\\
&=\sum_{\substack{d=1\\b=0}}^{\infty}d^{2g}\stirling{d+b-2}{d-2}((d-1)!)^{1-\chi}x^d\hbar^{b+d(1-\chi)}\\
&=\bigg(x\frac{\partial}{\partial x}\bigg)^{2g}x\hbar^{1-\chi}\sum_{\substack{d=1\\b=0}}^{\infty}\stirling{d+b-2}{d-2}((d-1)!)^{1-\chi}x^{d-1}\hbar^{b-(d-1)(1-\chi)}\\
&=\bigg(x\frac{\partial}{\partial x}\bigg)^{2g}x\hbar^{1-\chi}\bigg(1+\sum_{\substack{d=1\\b=0}}^{\infty}\stirling{d+b-1}{d-1}(d!)^{1-\chi}x^{d}\hbar^{b+d(1-\chi)}\bigg)\\
&=\frac{x}{\hbar}\bigg(\hbar\frac{\partial}{\partial x}x\bigg)^{2g}\bigg(1+\sum_{\substack{d=1\\b=0}}^{\infty}\stirling{d+b-1}{d-1}(d!)^{1-\chi}x^{d}\hbar^{b+d(1-\chi)}\bigg) \\
&=\frac{x}{\hbar}\bigg(\hbar\frac{\partial}{\partial x}x\bigg)^{2g}Z_\leq^g.
\end{align}
Putting things together and multiplying by $\frac{\hbar}{x}$ we get
\begin{align}
\bigg[x\hbar^2\frac{\partial}{\partial x}-\hbar\frac{\partial}{\partial x} +\bigg(\hbar\frac{\partial}{\partial x} x\bigg)^{2g}\bigg]Z_\leq^g=0.
\end{align}
Substituting $\widehat{x}=x$ and $\widehat{y}=-\hbar\frac{\partial}{\partial x}$ we obtain the claim
\begin{align}
[\widehat{x}\widehat{y}^2+\widehat{y}+(\widehat{y}\widehat{x})^{2g}]Z_\leq^g&=0.\qedhere
\end{align}
\end{proof}
For the strictly monotone Hurwitz number we have a similar result.

\begin{theorem}
\label{thm:strict}
The partition function $Z_<^g$ satisfies the differential equation
\begin{align}
[\widehat{y}+(1-\widehat{x}\widehat{y})(\widehat{y}\widehat{x})^{2g}]Z_<^g=0,
\end{align}
where $\widehat{x}=x$ and $\widehat{y}=-\hbar\frac{\partial}{\partial x}$.
\end{theorem}

\begin{proof}
Again we use the recursion formula of the Stirling numbers and \cref{PartitionFunktionStirling}. The recursion formula yields
\begin{align}
\fstirling{d}{d-b}=(d-1)\fstirling{d-1}{d-b}+\fstirling{d-1}{d-b-1}.
\end{align}
We multiply this equation by $(d!)^{1-\chi}dx^{d-1}\hbar^{b+d(1-\chi)+1}$ and sum over $d>0$ and $0\leq b\leq d-1$. For the left hand side we get
\begin{equation}
\sum_{d=1}^\infty\sum_{b=0}^{d-1}\fstirling{d}{d-b}(d!)^{1-\chi}dx^{d-1}\hbar^{b+d(1-\chi)+1}=\hbar\frac{\partial}{\partial x}Z_<^g.
\end{equation}
For the first term on the right hand side we get
\begin{align}
\sum_{d=1}^\infty\sum_{b=0}^{d-1}&\fstirling{d-1}{d-b}(d!)^{1-\chi}d(d-1)x^{d-1}\hbar^{b+d(1-\chi)+1}\\
&=\bigg(x\hbar\frac{\partial}{\partial x}\bigg)\sum_{d=1}^\infty\sum_{b=0}^{d-1}\fstirling{d-1}{d-b}(d!)^{1-\chi}dx^{d-1}\hbar^{b+d(1-\chi)}\\
&=\bigg(x\hbar\frac{\partial}{\partial x}\bigg)\sum_{d=1}^\infty\sum_{b=0}^{d-1}\fstirling{d-1}{d-b}(d-1)!^{1-\chi}d^{2g}\hbar^{2g}x^{d-1}\hbar^{(b-1)+(d-1)(1-\chi)}\\
&=\bigg(x\hbar\frac{\partial}{\partial x}\bigg)^{2g}\bigg(\hbar\frac{\partial}{\partial x}x\bigg)\sum_{d=1}^\infty\sum_{b=0}^{d-1}\fstirling{d-1}{d-b}(d-1)!^{1-\chi}x^{d-1}\hbar^{(b-1)+(d-1)(1-\chi)}.
\end{align}
Note that for $b=0$ it holds that $\fstirling{d-1}{d}=0$, hence we continue by
\begin{align}
\bigg(x\hbar&\frac{\partial}{\partial x}\bigg)\bigg(\hbar\frac{\partial}{\partial x}x\bigg)^{2g}\sum_{d=1}^\infty\sum_{b=0}^{d-1}\fstirling{d-1}{d-b}(d-1)!^{1-\chi}x^{d-1}\hbar^{(b-1)+(d-1)(1-\chi)}\\
&=\bigg(x\hbar\frac{\partial}{\partial x}\bigg)\bigg(\hbar\frac{\partial}{\partial x}x\bigg)^{2g}\sum_{d=2}^\infty\sum_{b=1}^{d-1}\fstirling{d-1}{(d-1)-(b-1)}(d-1)!^{1-\chi}x^{d-1}\hbar^{(b-1)+(d-1)(1-\chi)}\\
&=\bigg(x\hbar\frac{\partial}{\partial x}\bigg)\bigg(\hbar\frac{\partial}{\partial x}x\bigg)^{2g}\sum_{d=1}^\infty\sum_{b=0}^{d-1}\fstirling{d}{d-b}(d)!^{1-\chi}x^{d}\hbar^{b+d(1-\chi)}\\
&=\bigg(x\hbar\frac{\partial}{\partial x}\bigg)\bigg(\hbar\frac{\partial}{\partial x}x\bigg)^{2g}Z_<^g.
\end{align}
For the second term on the right hand, we get
\begin{align}
\sum_{d=1}^\infty\sum_{b=0}^{d-1}&\fstirling{d-1}{d-b-1}(d!)^{1-\chi}d x^{d-1}\hbar^{b+d(1-\chi)+1}\\
&=\bigg(\hbar\frac{\partial}{\partial x}x\bigg)^{2g}\sum_{d=1}^\infty\sum_{b=0}^{d-1}\fstirling{d-1}{d-b-1}(d-1)!^{1-\chi}x^{d-1}\hbar^{b+(d-1)(1-\chi)}\\
&=\bigg(\hbar\frac{\partial}{\partial x}x\bigg)^{2g}\sum_{d=2}^\infty\sum_{b=0}^{d-1}\fstirling{d-1}{d-b-1}(d-1)!^{1-\chi}x^{d-1}\hbar^{b+(d-1)(1-\chi)}\\
&=\bigg(\hbar\frac{\partial}{\partial x}x\bigg)^{2g}Z_<^g,
\end{align}
where in the second equality we used that for $d=1,b=0$ the Stirling number vanishes. Finally we put these equations together and obtain
\begin{align}
\bigg[-\hbar\frac{\partial}{\partial x}+\bigg(1-x\bigg(-\hbar\frac{\partial}{\partial x}\bigg)\bigg)\bigg(-\hbar\frac{\partial}{\partial x}x\bigg)^{2g}\bigg]Z_<^g=0.
\end{align}
Substituting $\widehat{x}=x$ and $\widehat{y}=-\hbar\frac{\partial}{\partial x}$ we obtain the claim
\begin{align}
[\widehat{y}+(1-\widehat{x}\widehat{y})(\widehat{y}\widehat{x})^{2g}]Z_<^g&=0. \qedhere
\end{align}
\end{proof}

\begin{remark}
Note, that for $g=0$, we recover the quantum curve for the usual monotone (strictly monotone) Hurwitz numbers of \cite{DDMmonotone} (resp. \cite{DMstrictly}).
\end{remark}

\section{A mysterious topological recursion}
\label{sec:myst}
In this section, we consider the quantum curve of the strictly monotone Hurwitz numbers derived in \cref{thm:strict}. We focus on the case with elliptic base curve, i.e. on $g=1$. Computing the semi-classical limit, we obtain the spectral curve
\begin{equation}
y+(1-xy)(xy)^2
\end{equation}
with parametrisation
\begin{equation}
x(z)=\frac{(z-1)^2}{2},\quad y(z)=\frac{z}{(z-1)^3}.
\end{equation}
Surpringly, running topological recursion for this input data yields the monotone single Hurwitz numbers up to a combinatorial pre-factor. More precisely, we obtain the cumulants of the Weingarten function. This points towards an unknown relationship between the combinatorics of strictly monotone Hurwitz numebrs with elliptic base curve and monotone Hurwitz numbers with rational base curve.\par 
As already noted in the introduction, topological recursion for monotone single Hurwitz numbers for a different normalisation was proved in \cite{DDMmonotone}, however for a different spectral curve and the exclusion of the $(0,2)-$case. In our case, the $(0,2)$ case still encodes the relevant invariants. We begin by defining the correct normalisation of monotone single Hurwitz numbers for our purpose, which coincide with the cumulants of the Weingarten function. The latter was motivated by discussions with James Mingo on problems of higher order freeness in free probability \cite{collins2007second}. In particular using the following normalizations, the numbers coincide with certain values of the M\"obius function on the set of partitioned permutations, this oberservation is still not perfectly understood and will be investigated in future work.

\begin{definition}
Let $g$ be a non negative integer, $d,n$ be a positive integers and $\mu$ a partition $d$ of length $l$. We denote by $m_{g,n}(\mu)$ the number of connected labeled monotone factorizations of a fixed (but arbitrary) permutation $\sigma$ with $C(\sigma)=\mu$. Moreover we put
\begin{align}
C_{g,n}(\mu)=(-1)^{2g-2+n+|\mu|}m_{g,n}(\mu)=(-1)^{n+|\mu|}m_{g,n}(\mu)
\end{align}
and denote by $W_{g,n}(x_1,\dots, x_n)$ the corresponding generating series, i.e.
\begin{align}
W_{g,n}(x_1,\dots, x_n)=\sum_{\mu_1,\dots,\mu_n=1}^\infty \frac{C_{g,n}(\mu_1,\dots,\mu_n)}{x_1^{\mu_1+1}\dots x_n^{\mu_n+1}}.
\end{align}
\end{definition}

\begin{remark}
\begin{enumerate}
\item
The numbers $C_{g,n}(\mu)$ agree with the monotone Hurwitz numbers up to the combinatorial factor $(-1)^{n+|\mu|}\prod_{i=1}^l\mu_i$.

\item
When we drop the connectivity condition in the definition of $m_{g,n}(\mu)$, we analogously, obtain the \text{disconnected} analogs of $C_{g,n}(\mu)$.  These numbers are the coefficients of the asymptotic expansion of the Weingarten function \cite{GGPNmonotone}.

\item
In \cite{GGPNzero}, the numbers $m_{g,n}(\mu)$ are put into a generating series via
\begin{align}
M_{g,n}(y_1,\dots,y_n)=\sum_{\mu_1,\dots,\mu_n=1}^{\infty}m_{g,n}(\mu_1,\dots,\mu_n)y_1^{\mu_1-1}\dots y_n^{\mu_n-1}.
\end{align}
Thus their generating series relates to $W_{g,n}(x_1,\dots,x_n)$ as follows
\begin{align}
W_{g,n}(x_1,\dots,x_n)=\frac{M_{g,n}(\frac{-1}{x_1},\dots,\frac{-1}{x_n})}{x_1^2\dots x_n^2}.
\end{align}
\end{enumerate}
\end{remark}

The following is the main theorem of this section, which we prove in \cref{sub:proofTR}.

\begin{theorem}\label{TheoremTR}
The numbers $C_{g,n}(\mu)$ satisfy topological recursion with the spectral curve given by
\begin{align}
x(z)=\frac{(z-1)^2}{z},\quad y(z)=\frac{z}{(z-1)^3},
\end{align}
i.e. the differentials
\begin{align}
\omega_{g,n}(z_1,\dots,z_n)&=\sum_{\mu_1,\dots,\mu_n=1}^\infty \frac{C_{g,n}(\mu_1,\dots,\mu_n)}{x(z_1)^{\mu_1+1}\dots x(z_n)^{\mu_n+1}} \diff x(z_1)\dots \diff x(z_n)\quad\textrm{for}\quad (g,n)\neq (0,2)\\
\omega_{0,2}&=\sum_{\mu_1,\mu_2=1}^\infty\frac{C_{0,2}(\mu_1,\mu_2)}{x(z_1)^{\mu_1+1}x(z_2)^{\mu_2+1}}dx(z_1)\diff x(z_2)+\frac{\diff x(z_1)dx(z_2)}{(x(z_1)-x(z_2))^2}
\end{align}
satisfy the recursion
\begin{align}\label{eq:TR}
\omega_{g,n}(z_1,\dots,z_n)=&\mathrm{Res}_{z=\pm 1} K(z_1,z)\bigg[\omega_{g-1,n+1}(z,\sigma(z),z_2,\dots,z_n)\\
&+\sum_{\substack{g_1+g_2=g\\ I\sqcup J= N\setminus \{1\}}}^\prime\omega_{g_1,|I|+1}(z,z_I)\omega_{g_2,|J|+1}(\sigma(z),z_J)\bigg]
\end{align}
on $2g-2+n > 0$, where
\begin{align}
K(z_1,z)=\frac{\frac{1}{2}\int_{\sigma(z)}^z \omega(z_1,\cdot)}{\omega_{0,1}(z)-\omega_{0,1}(\sigma(z))}=\frac{z(z-1)^3\diff z_1}{2(z+1)(z_1-z)(z_1z-1)},\quad \sigma(z)=\frac{1}{z}
\end{align}
and with the initial data iven by
\begin{align}
\omega_{0,1}(z)=y\diff x,\quad\text{and}\quad\omega_{0,2}(z_1,z_2)=B(z_1,z_2)\coloneqq \frac{\diff z_1 \diff z_2}{(z_1-z_2)^2}.
\end{align}
\end{theorem}

\begin{remark}
Note that $\diff x(z)=\frac{z^2-1}{z^2}$ has the zeroes $z=\pm 1$ and since $y(z)$ has a pole of order bigger than 1 at $z=1$, the spectral curve $(x,y)$ is irregular (see \cite{DN2018topological}). Hence the invariants $\omega_{g,n}$ agree with the invariants obtained from the local spectral curve obtained by removing the point $z=1$. In particular the residue at $z=1$ in \cref{eq:TR} does not contribute and it suffices to compute the residue at $z=-1$.\par 
Moreover, we note that while $\omega_{0,2}\neq W_{0,2}\diff x(z_1)\diff x(z_2)$, we have
\begin{equation}
\Res_{z_1,z_2\to\infty}x(z_1)^{\mu_1}x(z_2)^{\mu_2}\omega_{0,2}(z_1,z_2)=C_{0,2}(\mu_1,\mu_2)
\end{equation}
since their difference is holomorphic by definition.
\end{remark}

The starting point of our proof is the following recursion, which is a direct consequence of \cite[Theorem 2.1]{GGPNzero}.

\begin{proposition}\label{cut&join}
Let $g$ be a non-negative integer, $n\in\mathbb{N}$ and $\mu=(\mu_1,\dots,\mu_n)$ a partition of a positive integer. Then we have the following recursion
\begin{align}
-C_{g,n}(\mu) &= \sum_{j=2}^n \mu_jC_{g,n-1}(\mu_1+\mu_j,\mu_{N\setminus\{1\}})  + \sum_{\alpha+\beta=\mu_1}C_{g-1,n+1}(\alpha,\beta,\mu_{N\setminus\{1\}}) \\
&\quad + \sum_{\alpha+\beta=\mu_1}\sum_{\substack{g_1+g_2=g \\ I\sqcup J = N\setminus\{1\}}}C_{g_1,1+|I|}(\alpha,\mu_I)C_{g_2,1+|J|}(\beta,\mu_J),
\end{align}
where $\mu_I=(\mu_{i_1},\dots,\mu_{i_k})$ for $I=\{i_1,\dots,i_k\}$, $N=\{2,\dots,n\}$ and the initial value $C_{0,1}(1)=1$.
\end{proposition}

The following proposition reformulates the cut and join equation \cref{cut&join} as a differential equation for generating series. 

\begin{proposition}
It holds that
\begin{align}
-W_{g,n}&(x_1,\dots,x_n )=\sum_{j=2}^n\frac{\partial}{\partial x_j}\frac{x_j}{x_1}\frac{x_1W_{g,n-1}(X_{\{1,\dots,n\}\setminus\{j\}})-x_jW_{g,n-1}(X_{\{1,\dots,n\}\setminus\{1\}})}{x_1-x_j}\\
&+\quad x_1 W_{g-1,n+1}(x_1,X_{\{1,\dots,n\}}) + x_1\sum_{\substack{g_1+g_2=g \\ I\sqcup J = N\setminus\{1\}}}W_{g_1,|I|+1}(x_1,X_I)W_{g_2,|J|+1}(x_1,X_J) - \frac{1}{x^2}\delta_{g,0}\delta_{n,1}
\end{align}
\end{proposition}

\begin{proof}
We multiply the cut and join equaions from \cref{cut&join} by $x_1^{-(\mu_1+1)}\dots x_n^{-(\mu_n+1)}$ and sum over $\mu_1,\mu_2\dots \mu_n\geq 1$. We start by dealing with the first term on the right hand side. First observe the following. For fixed $j$ we have

\begin{align}
\sum_{\mu_1,\mu_j =1}^\infty\mu_j \frac{C_{g,n-1}(\mu_1+\mu_j,\mu_{N\setminus\{\mu_j\}})}{x_1^{\mu_1+1} x_j^{\mu_j+1}} 
&=\sum_{\substack{\mu_1,\mu_j\geq 1}}\mu_j\frac{C_{g,n-1}(\mu_1+\mu_j,\mu_{N\setminus\{\mu_j\}})}{x_1^{\mu_1+1}x_j^{\mu_j+1}}\\
&=-\frac{\partial}{\partial x_j} \sum_{\substack{\mu_1,\mu_j\geq 1}}\frac{C_{g,n-1}(\mu_1+\mu_j,\mu_{N\setminus\{\mu_j\}})}{x_1^{\mu_1+1}x_j^{\mu_j}}\\
&=-\frac{\partial}{\partial x_j} \sum_{\substack{\mu_1\geq 1,\mu_j\geq 0}}\frac{C_{g,n-1}(\mu_1+\mu_j,\mu_{N\setminus\{\mu_j\}})}{x_1^{\mu_1+1}x_j^{\mu_j}}\\
&=-\frac{\partial}{\partial x_j}\frac{1}{x_1^2}\sum_{\substack{\mu_1\geq 1,\mu_j\geq 0}}\frac{C_{g,n-1}(\mu_1+\mu_j,\mu_{N\setminus\{\mu_j\}})}{x_1^{\mu_1-1}x_j^{\mu_j}}\\
&=-\frac{\partial}{\partial x_j}\frac{1}{x_1^2}\sum_{\substack{\mu_1\geq 0,\mu_j\geq 0}}\frac{C_{g,n-1}(\mu_1+\mu_j+1,\mu_{N\setminus\{\mu_j\}})}{x_1^{\mu_1}x_j^{\mu_j}}\\
&=-\frac{\partial}{\partial x_j}\frac{1}{x_1^2}\sum_{\nu=0}^\infty \sum_{\substack{\mu_1+\mu_j=\nu \\ \mu_1,\mu_j\geq 0}}\frac{C_{g,n-1}(\nu+1,\mu_{N\setminus\{\mu_j\}})}{x_1^{\mu_1}x_j^{\mu_j}}.
\end{align}

We note that

\begin{align}
\sum_{\mu_1+\mu_j=\nu}\frac{1}{x_1^{\mu_1}x_j^{\mu_j}} &=\frac{\frac{1}{x_1^{\nu+1}}-\frac{1}{x_j^{\nu+1}}}{\frac{1}{x_1}-\frac{1}{x_j}}\\
&=-x_1x_j\frac{\frac{1}{x_1^{\nu+1}}-\frac{1}{x_j^{\nu+1}}}{x_1-x_j}.
\end{align}

and hence we find

\begin{align}
-\frac{\partial}{\partial x_j}\frac{1}{x_1^2}\sum_{\nu=0}^\infty \sum_{\substack{\mu_1+\mu_j=\nu \\ \mu_1,\mu_j\geq 0}} \frac{C_{g,n-1}(\nu+1,\mu_{N\setminus\{\mu_j\}})}{x_1^{\mu_1}x_j^{\mu_j}}
&=\frac{\partial}{\partial x_j}\frac{x_j}{x_1}\sum_{\nu=0}^\infty \frac{C_{g,n-1}(\nu+1,\mu_{N\setminus\{\mu_j\}})(\frac{1}{x_1^{\nu+1}}-\frac{1}{x_j^{\nu+1}})}{x_1-x_j}\\
&=\frac{\partial}{\partial x_j}\frac{x_j}{x_1}\sum_{\nu=0}^\infty \frac{C_{g,n-1}(\nu+1,\mu_{N\setminus\{\mu_j\}})(\frac{x_1}{ x_1^{\nu+2}}-\frac{x_j}{x_j^{\nu+2}})}{x_1-x_j}.
\end{align}

Further, we can put this in the summation over all $\mu_1,\dots,\mu_n$ and we obtain

\begin{align}
\sum_{\mu_1,\dots\mu_n =1}^\infty\mu_j &\frac{C_{g,n-1}(\mu_1+\mu_j,\mu_{N\setminus\{\mu_j\}})}{x_1^{\mu_1+1}\dots x_n^{\mu_n+1}}\\ 
&=\sum_{\substack{\mu_i=1\\ i=2,\dots,n \\ i\neq j}}^\infty \frac{1}{x_2^{\mu_2+1}\dots x_n^{\mu_n+1}}\frac{\partial}{\partial x_j}\frac{x_j}{x_1}\sum_{\nu=0}^\infty \frac{C_{g,n-1}(\nu+1,\mu_{N\setminus\{\mu_j\}})(\frac{x_1}{ x_1^{\nu+2}}-\frac{x_j}{x_j^{\nu+2}})}{x_1-x_j}\\
&= \frac{\partial}{\partial x_j}\frac{x_j}{x_1} \frac{x_1\sum_{\nu}\sum_{\mu_i}\frac{C_{g,n-1}(\nu+1,\mu_{N\setminus\{\mu_j\}})}{ x_1^{\nu+2}x_2^{\mu_2+1}\dots x_n^{\mu_n+1}}-x_j\sum_{\nu}\sum_{\mu_i}\frac{C_{g,n-1}(\nu+1,\mu_{N\setminus\{\mu_j\}})}{x_2^{\mu_2+1}\dots x_n^{\mu_n+1}x_j^{\nu+2}}}{x_1-x_j}\\
&=\frac{\partial}{\partial x_j}\frac{x_j}{x_1} \frac{x_1W_{g,n-1}(X_{\{1,\dots,n\}\setminus\{j\}})-x_jW_{g,n-1}(X_{\{1,\dots,n\}\setminus\{1\}})}{x_1-x_j}
\end{align}

We proceed analogously for the second term and observe that

\begin{align}
\sum_{\mu_1,\dots,\mu_n = 1}^{\infty}\sum_{\alpha+\beta=\mu_1}\frac{C_{g-1,n+1}(\alpha,\beta,\mu_{N\setminus\{1\}})}{x_1^{\mu_1+1}\dots x_n^{\mu_n+1}} &= \sum_{\mu_2,\dots\mu_n,\alpha,\beta = 1}^\infty\frac{C_{g-1,n+1}(\alpha,\beta,\mu_{N\setminus\{1\}})}{x_1^{\alpha + \beta+1}\dots x_n^{\mu_n+1}}\\
&= x_1 \sum_{\mu_2,\dots\mu_n,\alpha,\beta = 1}^\infty\frac{C_{g-1,n+1}(\alpha,\beta,\mu_{N\setminus\{1\}})}{x_1^{\alpha+1} x_1^{\beta+1}\dots x_n^{\mu_n+1}}\\
&=x_1 W_{g-1,n+1}(x_1,X_{\{1,\dots,x_n\}}).
\end{align}

Finally, for the third term we obtain

\begin{align}
\sum_{\mu_1,\dots,\mu_n = 1}^{\infty}\sum_{\alpha+\beta=\mu_1}&\sum_{\substack{g_1+g_2=g \\ I\sqcup J = N\setminus\{1\}}}\frac{C_{g_1,1+|I|}(\alpha,\mu_I)C_{g_2,1+|J|}(\beta,\mu_J)}{x_1^{\mu_1+1}\dots x_n^{\mu_n+1}}\\
&=x_1\sum_{\substack{g_1+g_2=g \\ I\sqcup J = N\setminus\{1\}}}\sum_{\substack{\alpha,\mu_i=1\\ i\in I}}^{\infty}\frac{C_{g_1,1+|I|}(\alpha,\mu_I)}{x_1^{\alpha+1}\prod_{i\in I}x_i^{\mu_i+1}}\sum_{\substack{\alpha,\mu_i=1\\ i\in J}}^{\infty}\frac{C_{g_1,1+|I|}(\beta,\mu_I)}{x_1^{\beta+1}\prod_{i\in J}x_i^{\mu_i+1}}\\
&=x_1\sum_{\substack{g_1+g_2=g \\ I\sqcup J = N\setminus\{1\}}}W_{g_1,|I|+1}(x_1,X_I)W_{g_2,|J|+1}(x_1,X_J)
\end{align}

and therefore

\begin{align}
-W_{g,n}(x_1,&\dots,x_n )=\sum_{j=2}^n\frac{\partial}{\partial x_j}\frac{x_j}{x_1}\frac{x_1W_{g,n-1}(X_{\{1,\dots,n\}\setminus\{j\}})-x_jW_{g,n-1}(X_{\{1,\dots,n\}\setminus\{1\}})}{x_1-x_j}\\
&+ x_1 W_{g-1,n+1}(x_1,X_{\{1,\dots,x_n\}}) + x_1\sum_{\substack{g_1+g_2=g \\ I\sqcup J = N\setminus\{1\}}}W_{g_1,|I|+1}(x_1,X_I)W_{g_2,|J|+1}(x_1,X_J) - \frac{1}{x^2}\delta_{g,0}\delta_{n,1}
\end{align}
\end{proof}

In the perspective of topological recursion it is handy to rewrite the cut and join equation in way that $W_{g,n}(x)$ does not appear on the right hand side.

\begin{corollary}
It holds
\begin{align}
-(1+ & 2x_1W_{0,1}(x_1))W_{g,n}(x_1,\dots,x_n )=\sum_{j=2}^n\frac{\partial}{\partial x_j}\frac{x_j}{x_1}\frac{x_1W_{g,n-1}(X_{\{1,\dots,n\}\setminus\{j\}})-x_jW_{g,n-1}(X_{\{1,\dots,n\}\setminus\{1\}})}{x_1-x_j}\\
&+ x_1 W_{g-1,n+1}(x_1,X_{\{1,\dots,n\}}) + x_1\sum_{\substack{g_1+g_2=g \\ I\sqcup J = N\setminus\{1\}}}^\prime W_{g_1,|I|+1}(x_1,X_I)W_{g_2,|J|+1}(x_1,X_J) - \frac{1}{x^2}\delta_{g,0}\delta_{n,1},
\end{align}
where $\sum^\prime$ means that the cases $(g_1,I)=(0,\emptyset)$ or $(g_2,J)=(0,\emptyset)$ are excluded.
\end{corollary}

We now compute some special cases of $W_{g,n}$, which require special treatment in the CEO topological recursion. We have the following result for the first few values of $(g,n)$.

\begin{corollary}[{\cite[Theorem 1.1]{GGPNzero}\cite[Theorem 6.2]{GGPNpoly}}]
We have that
\begin{align}
C_{0,1}(\mu)&=(-1)^{\mu-1}\frac{1}{\mu}\binom{2\mu-2}{\mu-1},\quad C_{0,2}(\mu_1,\mu_2)=(-1)^{\mu_1+\mu_2}\frac{2\mu_1\mu_2}{\mu_1+\mu_2}\binom{2\mu_1-1}{\mu_1}\binom{2\mu_2-1}{\mu_2}\\
C_{0,3}(\mu_1,\mu_2,\mu_3)&=(-1)^{\mu_1+\mu_2+\mu_3-1}8\mu_1\binom{2\mu_1-1}{\mu_1}\mu_2\binom{2\mu_2-1}{\mu_2}\mu_3\binom{2\mu_3-1}{\mu_3}.
\end{align}
\end{corollary}

A straightforward calculation shows the following lemma.

\begin{lemma}
\label{lem:cas}
The following identities hold
\begin{align}
W_{0,1}(x(z))&=\frac{z}{(z-1)^3},\quad W_{0,2}(x(z_1),x(z_2))  =\frac{z_1^2 z_2^2}{(z_1^2-1)(z_2^2-1)(1-z_1z_2)^2}\\
W_{0,3}(x_1,x_2,x_3)&=\frac{8}{x_1^2x_2^2x_3^2(1+\frac{4}{x_1})^{\frac{3}{2}}(1+\frac{4}{x_2})^{\frac{3}{2}}(1+\frac{4}{x_3})^{\frac{3}{2}}}=\prod_{i=1}^3\frac{2}{x_i^2(1+\frac{4}{x_i})^{\frac{3}{2}}}=\prod_{i=1}^3\frac{2}{(z_i+1)^2}\frac{1}{x^\prime(z_i)}.
\end{align}
\end{lemma}

The next lemma is a key step towards the topological recursion for the numbers $C_{g,n}(\mu)$, as determining the difference between the Bergman kernel and the $(0,2)$ free energy is important for the input data of the topological recursion.

\begin{lemma}
\label{lem:02}
We have
\begin{align}
W_{0,2}(x(z_1),x(z_2))\textrm{d}x(z_1)\textrm{d}x(z_2)= \frac{\diff z_1\diff z_2}{(1-z_1z_2)^2}=\frac{\diff z_1\diff z_2}{(z_1-z_2)^2} - \frac{\diff x(z_1)\diff x(z_2)}{(x(z_1)-x(z_2))^2}
\end{align}
and in particular
\begin{align}
W_{0,2}(x(z_1),x(z_2))\textrm{d}x(z_1)\textrm{d}x(z_2)=-B(\frac{1}{z_1},z_2).
\end{align}
\end{lemma}

\begin{proof}
From the last proposition and $x^\prime (z_i)= \frac{z_i^2-1}{z_i^2}$ we obtain
\begin{align}
W_{0,2}(x(z_1),x(z_2)) = \frac{z_1^2}{(z_1^2-1)}\frac{z_2^2}{(z_2^2-1)}\frac{1}{(1-z_1z_2)^2} = \frac{1}{x^\prime(z_1)}\frac{1}{x^\prime(z_2)}\frac{1}{(1-z_1z_2)^2},
\end{align}
from which the first equality follows immediately. For the second one we first note that
\begin{align}
(x(z_1)-x(z_2))^2 &=\bigg(\frac{z_1^2+1}{z_1}-\frac{z_2^2+1}{z_2}\bigg)^2\\
&=\bigg(\frac{z_2(z_1^2+1)-z_1(z_2^2+1)}{z_1z_2}\bigg)^2\\
&=\bigg(\frac{z_2z_1^2+z_2-z_1z_2^2-z_1}{z_1z_2}\bigg)^2\\
&=\bigg(\frac{(1-z_1z_2)(z_2-z_1)}{z_1z_2}\bigg)^2\\
&=\frac{(1-z_1z_2)^2(z_1-z_2)^2}{z_1^2z_2^2},
\end{align}
hence
\begin{align}
\frac{1}{(z_1-z_2)^2} - \frac{x^\prime(z_1)x^\prime(z_2)}{(x(z_1)-x(z_2))^2}&=\frac{1}{(z_1-z_2)^2} - \frac{\frac{z_1^2-1}{z_1^2}\frac{z_2^2-1}{z_2^2}}{\frac{(1-z_1z_2)^2(z_1-z_2)^2}{z_1^2z_2^2}}\\
&=\frac{1}{(z_1-z_2)^2}-\frac{(z_1^2-1)(z_2^2-1)}{(1-z_1z_2)^2(z_1-z_2)^2}\\
&=\frac{(1-z_1z_2)^2-(z_1^2-1)(z_2^2-1)}{(1-z_1z_2)^2(z_1-z_2)^2}\\
&=\frac{(z_1-z_2)^2}{(1-z_1z_2)^2(z_1-z_2)^2}\\
&=\frac{1}{(1-z_1z_2)^2}.
\end{align}
\end{proof}

\subsection{Proof of \cref{TheoremTR}}
\label{sub:proofTR}
Our proof of \cref{TheoremTR} is inspired by the approach in \cite{DDMmonotone}. We begin by considering the case $(g,n)=(0,3)$, which requires an independent discussion.

\begin{lemma}
\label{lem:03}
The multidifferential $\omega_{0,3}$ satisfies the recursion in \cref{eq:TR}.
\end{lemma}

\begin{proof}
Recall that by \cref{lem:cas}, we have 
\begin{align}
W_{0,3}(x_1,x_2,x_3)&=\frac{8}{x_1^2x_2^2x_3^2(1+\frac{4}{x_1})^{\frac{3}{2}}(1+\frac{4}{x_2})^{\frac{3}{2}}(1+\frac{4}{x_3})^\frac{3}{2}}\\
&=\prod_{i=1}^3\frac{2}{x_i^2(1+\frac{4}{x_i})^{\frac{3}{2}}}.
\end{align}
We find
\begin{align}
W_{0,3}(z_1,z_2,z_3)&=\prod_{i=1}^3\frac{2}{x_i(z)^2(1+\frac{4}{x_i(z)})^{\frac{3}{2}}}\\
&=\prod_{i=1}^3\frac{2z_i^2(z_i-1)^3}{(z_i-1)^4(z_i+1)^3}\\
&=\prod_{i=1}^3\frac{2}{(z_i+1)^2}\frac{z_i^2}{(z_i+1)(z_i-1)}\\
&=\prod_{i=1}^3\frac{2}{(z_i+1)^2}\frac{1}{x^\prime(z_i)}.
\end{align} 
The recursion from topological recursion reads
\begin{align}
 \omega_{0,3}&(z_1,z_2,z_3)=\Res_{z\to -1} K(z_1,z) \bigg[w_{0,2}(z,z_2)w_{0,2}(\frac{1}{z},z_3)+w_{0,2}(z,z_3)w_{0,2}(\frac{1}{z},z_2)\bigg]\\
 &=\Res_{z\to -1}\frac{z(z-1)^3\diff z_1}{2(z+1)(z_1-z)(z_1z-1)\diff z}\bigg[\frac{\diff z\diff z_2}{(z-z_2)^2}\frac{\diff \frac{1}{z}\diff z_3}{(\frac{1}{z}-z_3)^2}+\frac{\diff z\diff z_3}{(z-z_3)^2}\frac{\diff \frac{1}{z}\diff z_2}{(\frac{1}{z}-z_2)^2}\bigg]\\
 &=\Res_{z\to -1}\frac{1}{z+1}\frac{-(z-1)^3\diff z_1}{2 z(z_1-z)(z_1z-1)\diff z}\bigg[\frac{1}{(z-z_2)^2(\frac{1}{z}-z_3)^2}+\frac{1}{(z-z_3)^2(\frac{1}{z}-z_2)^2}\bigg]\diff z\diff z_1\diff z_2\diff z_3,
\end{align}
which is of the form 
\begin{align}
\frac{f(z,z_1,z_2,z_3)\diff z}{(z+1)}\diff z_1\diff z_2\diff z_3
\end{align}
where $f$ is holomorphic in $z$ around $z=-1$. Hence we get
\begin{align}
\omega_{0,3}(z_1,z_2,z_3)=f(1,z_1,z_2,z_3)\diff z_1\diff z_2\diff z_3 = \frac{8\diff z_1\diff z_2 \diff z_3}{(z_1+1)^2(z_2+1)^2(z_3+1)^2},
\end{align}
which concludes the proof.
\end{proof}

 Recall the polynomiality result for monotone Hurwitz numbers.
\begin{theorem}[\cite{GGPNpoly}]
There are symmetric rational functions $\vec{P}_{g,h}$ such that
\begin{align}
\vec{H}_{g,n}(\mu_1,\dots,\mu_n)=\prod_{i=1}^{n}\binom{2\mu_i}{\mu_i}\vec{P}_{g,n}(\mu_1,\mu_2,\dots,\mu_n).
\end{align}
Moreover if $(g,n)\neq (0,1), (0,2)$ then $\vec{P}_{g,n}$ is a polynomial with rational coefficients of degree $3g-3+n$.
\end{theorem}

Since $C_{g,n}$ agree with $\vec{H}_{g,h}$ up to the factor $(-1)^b \prod_{i=1}^n \mu_i$, we immediately get
\begin{align}
C_{g,n}=(-1)^b\prod_{i=1}^{n}\mu_i\binom{2\mu_i}{\mu_i}\vec{P}_{g,n}(\mu_1,\mu_2,\dots,\mu_n).
\end{align}
Thus, for,

\begin{equation}
\vec{P}_{g,n}(\mu_1,\dots,\mu_n)=\sum_{\underline{a}=\underline{0}}^{\textrm{finite}}B_{g,n}(\underline{a})\prod\mu_i^{a_i}
\end{equation}

we can write our generating function as

\begin{align}
W_{g,n}(x_1,\dots,x_n) &=\sum_{\underline{a}=\underline{0}}^{\textrm{finite}}B_{g,n}(\underline{a})\prod_{i=1}^n\sum_{\mu_i=1}^\infty\mu_i^{a_i+1}\binom{2\mu_i}{\mu_i}\bigg(\frac{-1}{x_i}\bigg)^{\mu_i+1}\\
&=\sum_{\underline{a}=\underline{0}}^{\textrm{finite}}B_{g,n}(\underline{a})\prod_{i=1}^n f_{a_i}(x_i)
\end{align}
with
\begin{equation}
f_a(x)=\sum_{\mu=1}^\infty\mu^{a+1}\binom{2\mu}{\mu}\bigg(\frac{-1}{x}\bigg)^{\mu+1}.
\end{equation}

A careful analysis of the functions $f_a$ will give us the analytic properties of $W_{g,n}$.

\begin{lemma}
\label{cor:pole}
Let $(g,n)\neq (0,1),(0,2)$, then the functions $W_{g,n}(z_1,\dots,z_n)$ satisfy
\begin{align}
W_{g,n}(z_1,\dots,z_n)=-W_{g,n}(\sigma(z_1),z_2,\dots,z_n)=-W_{g,n}(\frac{1}{z_1},z_2,\dots,z_n).
\end{align}
Moreover they are rational functions in each $z_i$ having poles at $z_i=1$ and at $z_i=-1$.
\end{lemma}

\begin{proof}
Note that the functions $f_a$ satisfy the recursion

\begin{equation}\label{a}
f_a(x)=\sum_{\mu=1}^\infty \mu^{a+1}\binom{2\mu_i}{\mu_i}\bigg(\frac{-1}{x}\bigg)^{\mu+1}
=-\frac{\partial}{\partial x}xf_{a-1}(x),
\end{equation}

i.e.
\begin{align}
f_a(x)=(-\frac{\partial}{\partial x}x)^a f_0(x)
\end{align}
and
\begin{align}
f_0(x)=\sum_{n=1}^\infty \mu\binom{2\mu}{\mu}\bigg(\frac{-1}{x}\bigg)^{\mu+1}=\frac{2}{\sqrt{x}{\sqrt{x+4}}^3}.
\end{align}
In the variable z we get
\begin{align}
f_0(z)=\frac{2z^2}{(z-1)(z+1)^3}
\end{align}
and
\begin{align}
f_0(\sigma(z))=f_0(\frac{1}{z})=\frac{2\frac{1}{z^2}}{(\frac{1}{z}-1)(\frac{1}{z}+1)^3}=-\frac{2z^2}{(z-1)(z+1)^3}=-f_0(z).
\end{align}
We find by induction
\begin{align}
f_a(z)+f_a(\sigma(z))=\frac{\partial}{\partial x}x(f_{a-1}(z)+f_{a-1}(\sigma(z)))=0
\end{align}
and hence
\begin{align}
W_{g,n}(z_1,z_2,\dots,z_n)+W_{g,n}(\sigma(z_1),z_2,\dots,z_n)&=\sum_{\underline{a}=\underline{0}}^{\text{fin}}B_{g,n}(\underline{a})[f_{a_1}(z)+f_{a_1}(\sigma(z))]\prod_{i=2}^n f_{a_i}(x_i)\\
&=0
\end{align}
Moreover note that \cref{a} reads
\begin{align}
f_a(z)&=\left(\frac{-z^2}{z^2-1}\frac{\partial}{\partial z}\frac{(z-1)^2}{z}\right) f_{a-1}(z)
\end{align}
in the variable $z$. It follows by induction that $f_a$ is rational and has a pole of order $1$ at $z=1$ and a pole of order $2a+3$ of $z=-1$.
\end{proof}

The last result can be reformulated in terms of the forms $\omega_{g,n}$.

\begin{corollary}
\label{cor:sym}
For $(g,n)\neq(0,1),(0,2)$, the forms $\omega_{g,n}(z_1,\dots,z_n)$ are antisymmetric w.r.t. $\sigma$, i.e.
\begin{align}
\omega_{g,n}(z_1,\dots,z_n)=-\omega_{g,n}(\sigma(z_1),z_2,\dots, z_n)=-\omega_{g,n}\left(\frac{1}{z_1},z_2,\dots, z_n\right).
\end{align}
They only have poles at $z=\pm 1$, where the pole at $z=1$ has at most order 1.
\end{corollary}

\begin{proof}
The assertions follow from the \cref{cor:pole}.
\end{proof}

Now we are ready to prove \cref{TheoremTR}.
\begin{proof}[Proof of \cref{TheoremTR}]
The initial data is given by $(g,n)=(0,1),(0,2)$ and the case of $(g,n)=(0,3)$ was proved in \cref{lem:03}. Thus, in the following we assume that $(g,n)\neq(0,1),(0,2),(0,3)$.\par
The idea is to add the recursions for $W_{g,n}(z_1,z_2,\dots,z_n)$ and $W_{g,n}(\sigma(z_1),z_2,\dots,z_n)$ and proceed with a careful combinatorial analysis after substituting the identity in \cref{cor:pole}.  
\begin{itemize}
\item
Firstly, we note that the left hand side yields the following
\begin{align}
&-(1+2x_1W_{0,1}(z_1))W_{g,n}(z_1,z_2,\dots,z_n) + -(1+2x_1W_{0,1}(\sigma(z_1)))W_{g,n}(\sigma(z_1),z_2,\dots,z_n)\\
&=-(1+2x_1W_{0,1}(z_1))W_{g,n}(z_1,z_2,\dots,z_n) + (1+2x_1W_{0,1}(\sigma(z_1)))W_{g,n}(z_1,z_2,\dots,z_n)\\
&=-2x_1[W_{0,1}(z_1)-W_{0,1}(\sigma(z_1))]W_{g,n}(z_1,z_2,\dots,z_n).
\end{align}
\item
The first term on the right hand side equals the term
\begin{align}
\sum_{j=2}^n\frac{\partial}{\partial x_j}\frac{x_j}{x_1}\frac{x_1W_{g,n-1}(Z_{\{1,\dots,n\}\setminus\{j\}})-x_jW_{g,n-1}(Z_{\{1,\dots,n\}\setminus\{1\}})}{x_1-x_j},
\end{align}
whith the same expression but with $z_1$ replaced by $\sigma(z_1)$. First we note that, we can rewrite the derivatives by
\begin{align}
\frac{\partial}{\partial x} = \frac{z^2}{z^2-1}\frac{\partial}{\partial z}.
\end{align} 
Now we want to focus on what happens when we replace $x(z_1)$ by $x(\sigma(z_1))$, using $x(\sigma(z))=x(z)$ we get terms of the form
\begin{align}
\frac{z_j^2}{z_j^2-1}\frac{\partial}{\partial z_j}\frac{x(z_j)}{x(z1)}\frac{x(\sigma(z_1))W_{g,n-1}(\sigma(z_1),Z_{\{2,\dots,n\}\setminus\{j\}})-x_jW_{g,n-1}(Z_{\{1,\dots,n\}\setminus\{1\}})}{x_1-x_j}.
\end{align}
By our observations the latter is equivalent to
\begin{align}
\frac{z_j^2}{z_j^2-1}\frac{\partial}{\partial z_j}\frac{x(z_j)}{x(z1)}\frac{-x_1W_{g,n-1}(Z_{\{1,\dots,n\}\setminus\{j\}})-x_jW_{g,n-1}(Z_{\{1,\dots,n\}\setminus\{1\}})}{x_1-x_j}.
\end{align}
Thus $x_1W_{g,n-1}(Z_{\{1,\dots,n\}\setminus\{j\}})$ will cancel in the sum and we end up with the term
\begin{align}
-2 \sum_{j=2}^n\frac{z_j^2}{z_j^2-1}\frac{\partial}{\partial z_j}\frac{x(z_j)^2}{x(z1)}\frac{W_{g,n-1}(Z_{\{1,\dots,n\}\setminus\{1\}})}{x_1-x_j}
\end{align}
\item
The second term on the right hand side is
\begin{align}
x_1 \bigg( W_{g-1,n+1}(z_1,z_1,z_2,\dots ,z_n) &+ W_{g-1,n+1}(\sigma(z_1),\sigma(z_1),z_2,\dots,z_n) \bigg) \\
& = -2x_1W_{g-1,n+1}(z_1,\sigma(z_1),z_2,\dots, z_n).
\end{align}
\item
The third term is
\begin{align}
x_1\bigg[\sum_{\substack{g_1+g_2=g \\ I\sqcup J = N\setminus\{1\}}}W_{g_1,|I|+1}(z_1,Z_I)W_{g_2,|J|+1}(z_1,Z_J) +
\sum_{\substack{g_1+g_2=g \\ I\sqcup J = N\setminus\{1\}}}W_{g_1,|I|+1}(\sigma(z_1),Z_I)W_{g_2,|J|+1}(\sigma(z_1),Z_J)\\
+2\sum_{j=2}^n\bigg( W_{0,2}(z_1,z_j)W_{g,n-1}(z_1,Z_{\{{1,\dots,n\}}\setminus\{j\}}) + W_{0,2}(\sigma(z_1),z_j)W_{g,n-1}(\sigma(z_1),Z_{\{{1,\dots,n\}}\setminus\{j\}})\bigg)\bigg].
\end{align}
By \cref{lem:02} and \cref{cor:sym} this yields
\begin{align}
-2x_1\bigg[\sum_{\substack{g_1+g_2=g \\ I\sqcup J = N\setminus\{1\}}}W_{g_1,|I|+1}(z_1,Z_I)W_{g_2,|J|+1}(\sigma(z_1),Z_J)
+ \sum_{j=2}^n\bigg(\frac{B(\sigma(z_1),z_j)}{\diff x(z_1)\diff x(z_j)}W_{g,n-1}(z_1,Z_{\{{1,\dots,n\}}\setminus\{j\}})\\+\frac{B(z_1,z_j)}{\diff x(\sigma(z_1))\diff x(z_j)}W_{g,n-1}(\sigma(z_1),Z_{\{{1,\dots,n\}}\setminus\{j\}})\bigg)\bigg].
\end{align}
\end{itemize}
As $(g,n)\neq(0,1),(0,2),(0,3)$ we have $(g,n-1)\neq (0,1),(0,2)$. Thus $W_{g,n-1}dx(z_1)\cdots dx(z_{n-1})$ satisfies \cref{cor:sym}. Therefore, putting things together, dividing by $-2x_1[W_{0,1}(x(z_1))-W_{0,1}(x(\sigma(z_1)))]$ and multiplying with $\diff x_1\dots \diff x_n$, we obtain
\begin{align}
\omega_{g,n}&(z_1,\dots,z_n)=\frac{1}{(W_{0,1}(x(z_1))-W_{0,1}(x(\sigma(z_1)))\diff x_1}\bigg[\sum_{j=2}^n\frac{\diff x_1\diff x_1}{x_1^2}\frac{z_j^2}{z_j^2-1}\frac{\partial}{\partial z_j}x_j^2\frac{\omega_{g,n-1}(z_2,\dots,z_n)}{x_1-x_j}\\
&+\omega_{g-1,n+1}(z_1,\sigma(z_1),z_2,\dots, z_n) + \sum_{\substack{g_1+g_2=g \\ I\sqcup J = N\setminus\{1\}}}^\circ\omega_{g_1,|I|+1}(z_1,Z_I)\omega_{g_2,|J|+1}(\sigma(z_1),Z_J)\\
&+ \sum_{j=2}\omega_{0,2}(\sigma(z_1),z_j)\omega_{g,n-1}(z_1,Z_{\{1,\dots,n\}\setminus\{j\}})+\omega_{0,2}(z_1,z_j)\omega_{g,n-1}(\sigma(z_1),Z_{\{1,\dots,n\}\setminus\{j\}})\bigg]\\
&=\frac{1}{\omega_{0,1}(z_1)-\omega_{0,1}(\sigma(z_1))}\bigg[\sum_{j=2}^n\frac{\diff x_1\diff x_1}{x_1^2}\frac{z_j^2}{z_j^2-1}\frac{\partial}{\partial z_j}x_j^2\frac{\omega_{g,n-1}(z_2,\dots,z_n)}{x_1-x_j}\\
&+\omega_{g-1,n+1}(z_1,\sigma(z_1),z_2,\dots, z_n) + \sum_{\substack{g_1+g_2=g \\ I\sqcup J = N\setminus\{1\}}}^\prime\omega_{g_1,|I|+1}(z_1,Z_I)\omega_{g_2,|J|+1}(\sigma(z_1),Z_J)\bigg].
\end{align}
The $\diff x_1$ in the denominator originates from the trivial expansion $\frac{\diff x_1}{\diff x_1}$. The next step is to apply Cauchys formula and use the fact that the $\omega_{g,n}$ are rational forms (in particular in $z_1$) having only poles at $\pm 1$. We have
\begin{align}
\omega_{g,n}(z_1,\dots,z_n) &= \Res_{z\to z_1}\frac{\omega_{g,n}(z,z_2,\dots, z_n)\diff z_1}{z-z_1}\\
&=\Res_{z\to\pm 1} \frac{\diff z_1}{z_1-z}\omega_{g,n}(z,z_2,\dots, z_n)\\
&=\Res_{z\to\pm 1} \frac{\diff z_1}{z_1-\frac{1}{z}}\omega_{g,n}(\frac{1}{z},z_2,\dots, z_n)\\
&=-\Res_{z\to\pm 1} \frac{\diff z_1}{z_1-\sigma(z)}\omega_{g,n}(z,z_2,\dots, z_n),
\end{align}
the second equality is due to the fact that $\omega_{g,n}$ are rational differentials in each $z_i$, hence the sum over all residue must vanish, i.e.
\begin{align}
0=\Res_{z\to z_1}\frac{\omega_{g,n}(z,z_2,\dots, z_n)\diff z_1}{z-z_1} + \Res_{z\to \pm 1}\frac{\omega_{g,n}(z,z_2,\dots, z_n)\diff z_1}{z-z_1}
\end{align}
where $\Res_{z\to \pm 1}$ denotes the sum of the residues at $1$ and $-1$. Thus we get
\begin{align}
\omega_{g,n}(z_1,\dots,z_n)&=\Res_{z\to\pm 1}\frac{1}{2}\bigg[\frac{\diff z_1}{z_1-z}-\frac{\diff z_1}{z_1-\sigma(z)}\bigg]\omega_{g,n}(z,z_2,\dots, z_n)\\
&=\Res_{z\to\pm 1}\bigg[\frac{1}{2}\int_{\sigma(z)}^z\omega_{0,2}(z_1,\cdot)\bigg]\omega_{g,n}(z,z_2,\dots, z_n).
\end{align}
Now we want to invoke the recursion for the $\omega_{g,n}$ which we established before. We get
\begin{align}
\omega_{g,n}(z_1,\dots,z_n)&=\Res_{z\to\pm 1}\frac{\frac{1}{2}\int_{\sigma(z)}^z\omega_{0,2}(z_1,\cdot)}{\omega_{0,1}(z)-\omega_{0,1}(\sigma(z))}\bigg[\sum_{j=2}^n\frac{\diff x\diff x}{x^2}\frac{z_j^2}{z_j^2-1}\frac{\partial}{\partial z_j}x_j^2\frac{\omega_{g,n-1}(z_2,\dots,z_n)}{x-x_j}\\
& +\omega_{g-1,n+1}(z,\sigma(z),z_2,\dots, z_n) + \sum_{\substack{g_1+g_2=g \\ I\sqcup J = N\setminus\{1\}}}\omega_{g_1,|I|+1}(z,Z_I)\omega_{g_2,|J|+1}(\sigma(z),Z_J)\bigg].
\end{align}
First we need to argue that the residue at $z=1$ does not contribute. But this is since
\begin{align}
K(z_1,z)=\frac{\frac{1}{2}\int_{\sigma(z)}^z \omega(z_1,\cdot)}{\omega_{0,1}(z)-\omega_{0,1}(\sigma(z))}=\frac{z(z-1)^3\diff z_1}{2(z+1)(z_1-z)(z_1z-1)\diff z},
\end{align}
i.e. $K$ has a zero of order 3 at $z=1$ which cancels the poles (of order 1) of the $\omega_{g,n}(z,z_2,\dots,z_n)$.
Hence last two terms on the right hand side vanish. For the first one, note that
\begin{align}
\frac{\diff x \diff x}{x^2}=\frac{(z+1)^2\diff z \diff z}{(z-1)^2 z^2}
\end{align}
has pole of order 2, so the zero of $K(z_1,z)$ cancels this as well. Lastly we show that the first term on the right hand side vanishes if we take the residue at $z=-1$. But by the last equation we see that the pole of order 1 of $K$ is removed by the zero of order 2. Thus we finally arrive at
\begin{align}
\omega_{g,n}(z_1,\dots,z_n)&=\Res_{z\to -1}\frac{\frac{1}{2}\int_{\sigma(z)}^z\omega_{0,2}(z_1,\cdot)}{\omega_{0,1}(z)-\omega_{0,1}(\sigma(z))}\bigg[\omega_{g-1,n+1}(z,\sigma(z),z_2,\dots, z_n)\\
 &+ \sum_{\substack{g_1+g_2=g \\ I\sqcup J = N\setminus\{1\}}}\omega_{g_1,|I|+1}(z,Z_I)\omega_{g_2,|J|+1}(\sigma(z),Z_J)\bigg].
\end{align}
\end{proof}

\section{Recursions for coverings of an arbitrary base curve}
\label{sec:refine}
In this section, we derive a recursion for monotone and Grothendieck dessins d'enfants coverings of arbitrary base curves. For $m\in\mathbb{N}$, we denote $[m]\coloneq\{1,\dots,m\}$. Furthermore, for a partition $\mu$, we denote the partition obtained by removing the entries in position $i_1,\dots,i_n$ (for some $n\le\ell(\mu)$ and $i_j\in[\ell(\mu)]$) by $\mu[i_1,\dots,i_n]$.
We introduce the following notation
\begin{notation}
For $g\ge0$ and $\nu$ be a partition of a positive integer $d$ of length $n$, we define
\begin{equation}
A_{g}(\nu)\coloneqq\left|\{(\alpha_1,\beta_1,\dots,\alpha_{g},\beta_{g})\in S_d^{2g}\colon C\left([\alpha_1,\beta_1]\cdots[\alpha_{g},\beta_{g}]\right)=\nu\}\right|.
\end{equation}

For each partition $\nu$, we fix a permutation
\begin{equation}
\sigma_{\nu}=\left(1\cdots\nu_1\right)\left(\nu_1+1\cdots\nu_1+\nu_2\right)\cdots\left(\sum_{i=1}^{n-1}\nu_i+1\cdots d\right),
\end{equation}
where for $k\in [n]$ the cycle $\left(\sum_{i=1}^{k-1}\nu_i+1\cdots \sum_{i=1}^k\nu_i\right)$ is labeled $k$.
\end{notation}

We now observe
\begin{equation}
\vec{H}_{g',\le}^{\bullet,g}(\mu)=\sum_{\nu\vdash d}\frac{\vec{h}^{\le,\bullet}_{g'}(\mu,\nu)}{|\mathrm{Aut}(\nu)|}\cdot\frac{ A_{g}(\nu)}{|C_{\nu}|}\quad\mathrm{and}\quad\vec{H}_{g',<}^{\bullet,g}(\mu)=\sum_{\nu\vdash d}\frac{\vec{h}^{<,\bullet}_{g'}(\mu,\nu)}{|\mathrm{Aut}(\nu)|}\cdot\frac{ A_{g}(\nu)}{|C_{\nu}|}
\end{equation}
for the usual (strictly) monotone double Hurwitz number $h^{\le,\bullet}_{g'}(\mu,\nu)$ (resp. $h^{<,\bullet}_{g'}(\mu,\nu)$). Thus, by finding a recursive method for computing the (strictly) monotone double Hurwitz numbers in genus $0$, we obtain a recursive method for all genera.\par 
For the rest of this section, we derive a recursive structure for (strictly) monotone double Hurwitz numbers. We note that the disconnected and connected (strictly) monotone double Hurwitz numbers are equivalent by the usual inclusion-exclusion principle. We can therefore focus on the connected numbers, which yield more compact formulas.\par
While monotone single Hurwitz numbers satisfy a recursive formula, a recursive formula for monotone Hurwitz numbers with an additional non-trivial ramification remains an open problem. However, one approach to this problem was introduced in \cite{DKmonotone} in the case of monotone orbifold Hurwitz numbers. More precisely, instead of considering monotone orbifold Hurwitz numbers, one considers a combinatorial refinement, which determines the former numbers and satisfy a recursion. In this section, we follow this philosophy and introduce a combinatorial refinement of (strictly) monotone double Hurwitz numbers and derive a recursion.

\begin{definition}
\label{def:refine}
Let $\mu,\nu$ be ordered partitions of $d$, let $i\in[\ell(\mu)],l\in[\nu_{\ell(\nu)}]$. We define $N^{\le;l,i}_g(\mu_i\mid\mu[i],\nu)$ to be the number of all tuples $(\sigma_1,\tau_1,\dots,\tau_b,\sigma_2)$, such that
\begin{enumerate}
\item $\sigma_1=\sigma_{\nu}$, (with $\sigma_\nu$ and its labelling as defined above),
\item $c(\sigma_1)=\nu$, $c(\sigma_2)=\mu$,
\item the cycles of $\sigma_2,\sigma_2$ are labeled
\item $\tau_b\cdots\tau_1\sigma_1=\sigma_2$,
\item $\langle\sigma_1,\tau_1,\dots,\tau_b,\sigma_2\rangle$ is a transitive subgroup,
\item for $\tau_i=(s_i\,t_i)$ ,where $s_i<t_i$, we have $t_i\le t_{i+1}$,
\item we have $\tau_b=(s_b\, t_b)$, where $t_b=\sum_{j=1}^{\ell(\nu)-1}\nu_j+l$ for some $l\in[\ell(\nu)]$,
\item $t_b$ is contained in the $i-$th cycle of $\sigma_2$.
\end{enumerate}
We also define the numbers
\begin{equation}
N_g^{\le;i}(\mu_i\mid\mu[i],\nu)=\sum_{l=1}^{\nu_n}N^{\le;l,i}_g(\mu_i\mid\mu[i],\nu)
\end{equation}
and
\begin{equation}
N_g^{\le}(\mu,\nu)=\sum_{i=1}^{\ell(\mu)}N_g^{\le;i}(\mu_i\mid \mu[i],\nu).
\end{equation}
Similarly, we define the notions for the strictly monotone case and denote the respective numbers by $N^{<;l,i}_g(\mu_i\mid\mu[i],\nu)$, $N_g^{<;i}(\mu_i\mid\mu[i],\nu)$ and $N_g^{<}(\mu,\nu)$.
\end{definition}

\begin{remark}
Before we state our recursion for the numbers $N^{\le;l,i}_g(\mu_i\mid\mu[i],\nu)$ and $N^{<;l,i}_g(\mu_i\mid\mu[i],\nu)$, we make the following remarks.
\begin{itemize}
\item We note that we can compute the monotone Hurwitz numbers $\vec{h}^{\le}_g(\nu,\mu)$ by considering all tuples $(\sigma_1,\tau_1,\dots,\tau_b,\sigma_2)$ satisfying conditons (2)--(6) in \cref{def:refine} as changing the order of the cycle types in the factorisations does not alter the enumeration. We use this convention in the proof of \cref{thm:recur}.
\item We see that $\vec{h}^{\le}_g(\nu,\mu)=\frac{|C_{\nu}|}{d!}\sum_{i=1}^mN_g^{\le}(\mu_i\mid \mu[i],\nu)$ and $\vec{h}^{<}_g(\nu,\mu)=\frac{|C_{\nu}|}{d!}\sum_{i=1}^mN_g^{<}(\mu_i\mid\mu[i],\nu)$. 
\item As we are concerned with (strictly) monotone factorisations in \cref{def:refine}, the numbers $t_b,t_b+1,t_b+2,\dots,|\nu|$ actually appear consecutively in the $i-$th cycle of $\sigma_2$ (see \cite{DKmonotone, Hahnmonodromy}), i.e. this cycle is of the form \begin{equation}
(\dots t_b\,t_{b}+1\,t_{b}+2\dots|\nu|).
\end{equation}
\item The above definition of $N^{\le,l}_g(\mu_i\mid\mu[i],\nu)$ is motivated by the notions in \cite{DKmonotone,Hahnmonodromy}. While of technical nature, they are more natural when considered in terms of \textit{monodromy graphs}, which is an approach taken in the works of Do--Karev and the first author. In their language, the numbers $N^{\le,l}_g(\mu_i\mid\mu[i],\nu)$ counts all \textit{"monodromy graphs, where the unique bold out-end is labeled $i$ with counter $l$"}. We note that these monodromy graphs are related to tropical covers but differ from the ones, we introduced in this work.
\end{itemize} 
\end{remark}

\begin{theorem}
\label{thm:recur}
Let $g$ be a non-negative integer and $\mu,\nu$ partitions of the same positive integer. Moreover, let $i\in[\ell(\mu)]$ and $l\le\nu_i$. Then we have
\begin{align}
N^{\le;l,i}_g(\mu_i\mid\mu[i],\nu)=\,&\Theta(\mu_i+l-\nu_n-1)\sum_{j\in [\ell(\mu)]\backslash\{i\}}\sum_{p=1}^lN_g^{\le;p,\ell(\mu)-1}(\mu_i+\mu_j\mid\mu^{i,j}[\ell(\mu)-1],\nu)\\
&+\sum_{\alpha+\beta=\mu_i}\sum_{p=1}^l\beta N_{g-1}^{\le;p,i}(\alpha\mid(\mu[i],\beta),\nu)\\
&+\sum_{\alpha+\beta=\mu_i}\sum_{g_1+g_2=g}\sum_{\substack{I_1\cup I_2=[\ell(\mu)]\backslash\{i\}\\J\subset [n-1]}}\sum_{p=1}^l\beta N_{g_1}^{\le;p,|I_1|+1}(\alpha\mid(\mu_{I_1},\alpha)[|I_1|+1],\nu_{J^c})\\
&N_{g_2}^{\le;|I_2|+1}(\beta\mid(\mu_{I_2},\beta),\nu_J)
\end{align}
and
\begin{align}
N^{<;l,i}_g(\mu_i\mid\mu[i],\nu)=\,&\Theta(\mu_i+l-\nu_n-1)\sum_{j\in [\ell(\mu)]\backslash\{i\}}\sum_{p=1}^{l-1}N_g^{<;p,\ell(\mu)-1}(\mu_i+\mu_j\mid\mu^{i,j}[\ell(\mu)-1],\nu)\\
&+\sum_{\alpha+\beta=\mu_i}\sum_{p=1}^{l-1}\beta N_{g-1}^{<;p,i}(\alpha\mid(\mu[i],\beta),\nu)\\
&+\sum_{\alpha+\beta=\mu_i}\sum_{g_1+g_2=g}\sum_{\substack{I_1\cup I_2=[\ell(\mu)]\backslash\{i\}\\J\subset [n-1]}}\sum_{p=1}^{l-1}\beta N_{g_1}^{<;p,|I_1|+1}(\alpha\mid(\mu_{I_1},\alpha)[|I_1|+1],\nu_{J^c})\\
&\left(\sum_{p=1}^{\nu_{\mathrm{max}(J)}}\left(N_{g_2}^{<;p,|I_2|+1}(\beta\mid(\mu_{I_2},\beta),\nu_J)\right)\right),
\end{align}
where $\Theta$ is the Heavyside step function, i.e. $\Theta(t)=0$ for $t<0$ and $\Theta(t)=1$ for $t\ge0$. Furthermore, we denote $\mu^{i,j}=\mu[i,j]\cup\{\mu_i+\mu_j\}$
\end{theorem}

\begin{proof}
As the proofs are completely parallel, we restrict our discussion to the monotone numbers. We prove this theorem by a cut-and-join analysis of transitive monotone factorisations $(\sigma_{\nu},\tau_1,\dots,\tau_b,\sigma)$ of type $(g,\nu,\mu)$. Recall that $N^{\le;l,i}_g(\mu_i\mid\mu[i],\nu)$ counts the number of all tuples $(\sigma_1,\tau_1,\dots,\tau_b,\sigma_2)$ of type $(g,\nu,\mu)$ satisfying conditions (1)--(8) \cref{def:refine}, where $\tau_i=(s_i\, t_i)$ with $s_i<t_i$, such that $\sigma_1=\sigma_{\nu}$ and $s_b$ is contained in the cycle of $\sigma_2$ labeled $i$. We fix such a factorisation, define
\begin{equation}
\Sigma=\tau_{b-1}\cdots\tau_1\sigma_{\nu}
\end{equation}
and observe that $\eta=(\sigma_{\nu},\tau_1,\dots,\tau_{b-1},\Sigma)$ is a monotone factorisation as well. There are three cases:
\begin{enumerate}
\item \textbf{The transposition $\tau_b$ is a cut} for $\Sigma$, i.e. $s_b$ and $t_b$ are contained in the same cycle of $\Sigma$. Then left multiplication by $\tau_b$, i.e. $\tau_b\Sigma$ cuts the cycle of $\Sigma$ containing $s_b$ and $t_b$ into two cycles. Conversely, this means $s_b$ and $t_b$ are contained in different cycles of $\sigma$ and $\tau_b$ joins those cycles by left multiplication $\tau_b\sigma_2$ to a joint cycle in $\Sigma$.\par
Thus, if $\tau_b$ is a cut for $\Sigma$, any such transitive monotone factorisation $(\sigma_{\nu},\tau_1,\dots,\tau_{b},\sigma)$ contributing to $N^{\le;l,i}_g(\mu_i\mid\mu[i],\nu)$  yields a transitive monotone factorisation $(\sigma_{\nu},\tau_1,\dots,\allowbreak\tau_{b-1},\Sigma)$ contributing to $N_g^{\le;p,\ell(\mu)-1}(\mu_i+\mu_j\mid\mu^{i,j}[\ell(\mu)-1],\nu)$ for some $j\in[\ell(\mu)]\backslash\{i\}$ and $p\le l$. The other way round, we start with a tuple $(\sigma_{\nu},\tau_1,\dots,\tau_{b-1},\Sigma)$ contributing to $N_g^{\le;p,i}(\mu_i+\mu_j\mid\mu[i,j],\nu)$ for some $j\in[\ell(\mu)]\backslash\{i\}$ and $p\le l$. We analyse the number of possible transpositions $\tau_b$ which give a transitive monotone factorisation $(\sigma_{\nu},\tau_1,\dots,\tau_b,\sigma)$ with $\mathcal{C}(\sigma_2)=\mu$. The number $s_b$ is fixed to be $\sum_{i=1}^{n-1}\nu_i+l$. As mentioned before, the cycle of $\sigma_2$ containing $s_b$ (labeled $i$) must then be of shape
\begin{equation}
\left(\cdots t_b\, t_b+1\cdots\sum_{k=1}^n\nu_i-1\,\sum_{k=1}^n\nu_i\right).
\end{equation}
Thus the length $\mu_i$ of the cycle must be at least $\nu_n-l+1$, in other words $\Theta(\mu_i+l-\nu_n-1)$ must not vanish. Moreover, as $t_b$ is fixed and $s_b$ must be contained in the same cycle of $\Sigma$ as $t_b$, the value of $s_b$ is fixed as well. This yields the first summand.
\item \textbf{The transposition $\tau_b$ is a redundant join} for $\Sigma$, i.e. $r_b$ and $s_b$ are contained in two different cycles of $\Sigma$ and $\eta$ is a transitive monotone factorisation as well. Then left multipliciation by $\tau_b$, i.e. $\tau_b\Sigma$ joins the cycles of $\Sigma$ containing $s_b$ and $t_b$ respectively to one cycle. Conversely, this means $s_b$ and $t_b$ are contained in the same cycle of $\sigma_2$ and $\tau_b$ cuts this cycle by left multiplication $\tau_b\sigma_2$ to two cycles of $\Sigma$.\par
Thus, if $\tau_b$ is a redundant join for $\Sigma$ any such tuple $(\sigma_{\nu},\tau_1,\allowbreak\dots,\tau_b,\sigma_2)$ contributing to $N^{\le;l,i}_g(\mu_i\mid\mu[i],\nu)$ yields a tuple $(\sigma_{\nu},\tau_1,\dots,\allowbreak\tau_{b-1},\Sigma)$ contributing to
$N_{g-1}^{\le;p,i}(\alpha\mid(\mu[i],\beta),\nu)$ for some $i\in [\ell(\mu)]$, $\alpha+\beta=\mu_i$ (note that the genus drops by $1$ since the number of transpositions and the length of the second cycle type drop by $1$) and $p\ge l$. The other way, we start with a tuple $(\sigma_{\nu},\tau_1,\dots,\tau_{b-1},\Sigma)$ contributing to $ N_{g-1}^{\le;p,i}(\alpha\mid(\mu[i],\beta),\nu)$ for some $i\in [\ell(\mu)]$ and $\alpha+\beta=\mu_i$. We analyse the number of possible transposition $\tau_b$ which yield a tuple $(\sigma_{\nu},\tau_1,\dots,\tau_b,\sigma_2)$ contributing to $N^{\le;l,i}_g(\mu_i\mid\mu[i],\nu)$.\par 
The number $s_b$ is fixed to be $\sum_{i=1}^{n-1}\nu_i+l$. By definition $t_b$ is contained in a cycle of length $\alpha$. We need to choose $s_b$, such that $\tau_b=(s_b\,t_b)$ joint $\alpha$ with another cycle to a new cycle of length $\mu_i$. Thus, we need to choose $r_b$ from a cycle of length $\beta$. Moreover, we can choose $s_b$ arbitrarily, which yields a factor of $\beta$ and we obtain the second summand.
\item \textbf{The transposition $\tau_b$ is an essential join} for $\Sigma$, i.e. $s_b$ and $t_b$ are contained in two different cycles of $\Sigma$ and the group generated by $\eta$ is non-transitive with two orbits. Then left multiplication by $\tau_b$, i.e. $\tau_b\Sigma$ joins two cycles of $\Sigma$ containing $s_b$ and $t_b$ respectively to one cycle. Conversely, this means $s_b$ and $t_b$ are contained in the same cycle of $\sigma_2$ and $\tau_b$ cuts this cycle by left multiplication $\tau_b\sigma_2$ to two cycles of $\Sigma$.\par 
Thus, if $\tau_b$ is an essential join for $\Sigma$ any such tuple $(\sigma_{\nu},\allowbreak\tau_1,\dots,\tau_b,\sigma_2)$ contributing to $N^{\le;l,i}_g(\mu_i\mid\mu[i],\nu)$ yields two tuples $(\sigma_{\nu_J},\allowbreak\tau_{i(1)},\dots,\tau_{i(c)},\sigma_2^1)$ and $(\sigma_{\nu_{J^c}},\tau_{j(1)},\dots,\tau_{j(d)},\sigma_2^2)$ for 
\begin{enumerate}
\item $J\subset[n]$,
\item $c+d=b$,
\item $\sigma_{\nu_J}$ is the permutation obtained as the product of all disjoint cycles of $\sigma_{\nu}$ with label in $J$ and
\item $\sigma_2^1$ and $\sigma_2^2$ have disjoint orbits and $\sigma_2=\sigma_2^1\sigma_2^2$, where $\sigma_2^i$ recovers the labels from $\sigma_2$ for $i=1,2$.
\end{enumerate}
These tupel contribute to $\sum_{p=1}^l\beta N_{g_1}^{\le;p,|I_1|+1}(\alpha\mid(\mu_{I_1},\alpha)[|I_1|+1],\nu_{J^c})$ for $p\le l$ and to $N_{g_2}^{\le;|I_2|+1}(\beta\mid(\mu_{I_2},\beta),\nu_J)$ respectively. By the same arguments as in the previous case we obtain a factor of $\beta$ and thus obtain the third summand, which completes the proof of the theorem. \qedhere
\end{enumerate}
\end{proof}

\appendix
\section{Examples}\label{sec:examples}
Here we provide some examples of the quasimodular $q$-expansions of $\sum_d H^{g,d}_{k,l,m}(\mu) q^d $, computed using Sage \cite{sagemath}. See \cite[Section 10]{RY10} for an extensive list of quasimodular forms corresponding to simple Hurwitz numbers ($l=m=0, \mu=()$). 
{\scriptsize
\begin{align}
\sum_d H^{2,d}_{2,0,0}() q^d &= \frac{1}{2^63^45}\left( 5 P^3-3 P Q-2 R\right) \\
&=2q^{2} + 16q^{3} + 60q^{4} + 160q^{5} + 360q^{6} + 672q^{7} + 1240q^{8} + 1920q^{9} + 3180q^{10} + 4400q^{11} + 6832q^{12} + O(q^{13})\\[5pt]
\sum_d H^{2,d}_{0,2,0}() q^d &= \frac{1}{2^73^45}\left( 5 P^3-3 P Q-2 R+45 P^2+18 Q+90 P-153\right)\\
&=2q^{2} + 13q^{3} + 44q^{4} + 109q^{5} + 235q^{6} + 422q^{7} + 760q^{8} + 1151q^{9} + 1875q^{10} + 2555q^{11} + 3927q^{12} + O(q^{13})\\[5pt]
\sum_d H^{2,d}_{0,0,2}() q^d &= \frac{1}{2^73^45}\left( 5 P^3-3 P Q-2 R-45 P^2-18 Q-90 P+153\right)\\
&=3q^{3} + 16q^{4} + 51q^{5} + 125q^{6} + 250q^{7} + 480q^{8} + 769q^{9} + 1305q^{10} + 1845q^{11} + 2905q^{12} + O(q^{13})
\end{align}
\begin{align}\begin{split}
\sum_d H^{3,d}_{2,0,0}(3) q^d &=\frac{1}{2^{13}3^55^2}\Big( -875 P^5+1775 P^3 Q-10 P^2 R-894 P Q^2+4 Q R+\\&\quad\quad\quad\quad 750 P^4-1710 P^2 Q+60 P R+900 Q^2+135 P^3-81 P Q-54 R\Big)\\
&=36 q^{3} + 540 q^{4} + 3606 q^{5} + 15726 q^{6} + 53298 q^{7} + 149142 q^{8} + 367920 q^{9} + 815886 q^{10} + 1668150 q^{11} + 3202374 q^{12} + O(q^{13})\\[5pt]
\sum_d H^{3,d}_{0,2,0}(3) q^d &= \frac{1}{2^{14}3^55^2}\Big( -875 P^5+1775 P^3 Q-10 P^2 R-894 P Q^2+4 Q R+
						\\&\quad\quad\quad\quad-2625  P^4-630  P^2 Q+\tfrac{23460}{7} P R+\tfrac{648}{7} Q^2+2835  P^3+1359 P Q-5814R
						 +3150 P^2-4608 Q+7020 P-4131\Big)\\
&=27 q^{3} + 369 q^{4} + 2337 q^{5} + 9795 q^{6} + 32307 q^{7} + 88446 q^{8} + 214536 q^{9} + 469230 q^{10} + 948600 q^{11} + 1803375 q^{12} + O(q^{13})\\[5pt]
\sum_d H^{3,d}_{0,0,2}(3) q^d &= \frac{1}{2^{14}3^55^2}\Big( -875 P^5+1775 P^3 Q-10 P^2 R-894 P Q^2+4 Q R+
						\\&\quad\quad\quad\quad+4125  P^4-2790  P^2 Q+\tfrac{22620}{7} P R+\tfrac{11952}{7} Q^2-2565  P^3-1521 P Q+5706 R+
						-3150 P^2+4608 Q-7020 P+4131\Big)\\	
&=9 q^{3} + 171 q^{4} + 1269 q^{5} + 5931 q^{6} + 20991 q^{7} + 60696 q^{8} + 153384 q^{9} + 346656 q^{10} + 719550 q^{11} + 1398999 q^{12} + O(q^{13})
\end{split}
\end{align}
}

\bibliographystyle{alpha}
\bibliography{literature.bib}
\end{document}